\documentclass[12pt]{amsart}
\usepackage{mathrsfs}
\usepackage{amsmath}
\usepackage{latexsym}
\usepackage{amssymb}
\usepackage{amsfonts}
\usepackage{multirow}
\usepackage{float}
\usepackage{enumerate}
\include{PDF}
\usepackage{graphics}
\usepackage{todonotes}
\usepackage{color}

\def\l{\langle} \def\r{\rangle} 
 
\def\FF{\mathbb F} \def\ZZ{\mathbb Z}

\def\mod{{\sf mod~}}  
 \def\diag{{\sf diag}}

 \def\soc{{\sf soc}}

\def\Cay{{\sf Cay}} 
\def\D{{\rm D}} \def\Q{{\rm Q}}

\def\C{{\bf C}}\def\N{{\bf N}}\def\Z{{\bf Z}} 

\def\Ome{{\it \Omega}}

\def\Del{{\it\Delta}}

\def\PSp{{\rm PSp}}

\def\A{{\rm A}}\def\Sym{{\rm Sym}}

\def\PSL{{\rm PSL}}\def\PGL{{\rm PGL}}
 \def\SL{{\rm SL}}
\def\AGL{{\rm AGL}}

 \def\PSU{{\rm PSU}} 
\def\Sz{{\rm Sz}}  
 
  \def\D{{\rm D}}

  \def\Sym{{\rm Sym}}
\def\C{{\bf C}}  
\def\D{\mathrm{D}} \def\Q{\mathrm{Q}} \def\A{\mathrm{A}} \def\Sy{\mathrm{S}}      \def\Sz{\mathrm{Sz}}
 
   \def\SL{\mathrm{SL}}     \def\PGL{\mathrm{PGL}} \def\PSL{\mathrm{PSL}} \def\PSU{\mathrm{PSU}} \def\PSp{\mathrm{PSp}}  \def\PGU{\mathrm{PGU}}     \def\AGL{\mathrm{AGL}}

\def\o{\omega}
\def\s{\sigma}
\def\d{\delta}

\newcommand{\bb}{\mathbb}

\def\le{\leqslant}
\def\ge{\geqslant}

\newtheorem{theorem}{Theorem}[section]
\newtheorem{definition}{Definition}[section]

\newtheorem{proposition}[theorem]{Proposition}
\newtheorem{lemma}[theorem]{Lemma}
\newtheorem{corollary}[theorem]{Corollary}
\newtheorem{example}[theorem]{Example}

\newtheorem{conjecture}[theorem]{Conjecture}
\newtheorem{problem}[theorem]{Problem}
\newtheorem{construction}[theorem]{Construction}
\newtheorem*{theorem*}{Theorem}

\def\qed{{\hfill$\Box$\bigskip}
}

\begin{document}

	\title{intersecting subsets in finite permutation groups}
	\thanks{This work was partially supported  by NSFC grants 11931005 and 61771019}
	
	\author[Li]{C. H. Li}
	\address{SUSTech International Center for Mathematics\\
		Department of Mathematics\\
		Southern University of Science and Technology\\
		Shenzhen, Guangdong 518055\\
		P. R. China}
	\email{lich@sustech.edu.cn}
	
	\author[Pantangi]{V. R. T. Pantangi}
	\address{Department of Mathematics\\
		University of Lethbridge \\
		Lethbridge \\
		Alberta TIK3M4, Canada}
	\email{pvrt1990@gmail.com}
	
	\author[Song]{S. J. Song}
	\address{School of Mathematics\\
		Yantai University\\
		Yantai, Shandong}
	\email{shujiao.song@hotmail.com}
	
	\author[Xie]{Y. L. Xie}
	\address{Department of Mathematics\\
		Southern University of Science and Technology\\
		Shenzhen, Guangdong 518055\\
		P. R. China}
	\email{11930521@mail.sustech.edu.cn}

	\begin{abstract}
			Let $G\le\Sym(\Ome)$ be transitive, and let $S$ be an {\it intersecting subset}, namely, the ratio $xy^{-1}$ of any elements $x,y\in S$ fixes some point.
		An {\it EKR-type problem} is to characterize transitive groups $G\le\Sym(\Ome)$ such that any intersecting set is upper bounded by $|G_\o|$, where $\o\in\Omega$.
			A nice result of Meagher-Spiga-Tiep (2016) tells us that if $G$ is 2-transitive, then indeed $|S|\leqslant|G_\o|$.
			A natural next step would be to explore intersecting subsets for primitive groups and quasiprimitive groups.
			Our study in this paper shows that for quasiprimitive permutation groups, the size $|S|$ can be arbitrarily larger than $|G_\o|$. We conjecture that for quasiprimitve groups,  the upperbound for $|S|$ is $O(|G_\o||\Omega|^{1\over2})$.
			As a starting point, we prove that ${|S|/(|G_\o||\Omega|^{1\over2}})\leqslant{\sqrt2/2}$ for all quasiprimitive actions of the Suzuki groups $G=\Sz(q)$.  To show that our conjectured upper bound is tight, we provide  examples of groups for which ${|S|/(|G_\o||\Omega|^{1\over2}})$ is arbitrarily close to ${\sqrt2/2}$. As far as general transitive groups concerned, infinity families of examples produced show that the ratio ${|S|/(|G_\o||\Omega|^{1\over2}})$ can be arbitrarily large.
	\end{abstract}
	
	
	\date{}
	
	\maketitle
	
	\section{Introduction}
	Let $G$ be a finite transitive permutation group on $\Ome$.
	A pair of permutations $x,y \in G$ is said to {\it intersect} if
	\[\{(\o,\o^{x})| \ \o \in \Ome\}\cap \{(\o,\o^{y})| \ \o \in \Ome\}\not=\emptyset,\]
	in other words, the ratio $xy^{-1}$ fixes some point in $\Omega$.
	
	\begin{definition}
		{\rm
			A subset $S$ of a transitive permutation group $G$ is said to be an {\it intersecting subset} if any pair of elements in $S$ intersect.
		}
	\end{definition}
	
	Clearly, a point stabilizer $G_\o$ and its cosets are all intersecting subsets.
	An EKR-type problem in permutation group theory is to study the cardinality and the structure of intersecting subsets in transitive permutation groups, where `EKR' standards for Erd\H{o}s-Ko-Rado, refer to \cite{GM2009}.
	An intersecting subset is said to be {\it maximum} if it is an intersecting subset of the maximum possible size.
	
	A transitive permutation group $G$ on $\Ome$ has the {\it EKR property} if its point stabilizers are maximum intersecting subsets.
	There have been many papers published in literature that study groups with EKR property, refer to \cite{AM2014, CK2003,DF1977,GM2009,LM2004,LPSX2018,MS2011,MST2016} and references therein.
	In particular, it was shown in \cite{MST2016} that all $2$-transitive permutation groups have the EKR property.
	A recent result \cite{MSi2019} shows that the characteristic vector of any maximum intersecting subset of $2$-transitive groups is a linear combination of characteristic vectors of cosets of point stabilizers.
	However, the results in \cite{Raghu} show that a permutation group in general does not have the property.
	
	In this paper, we explore intersecting subsets of permutation groups in more general case.
	It is not hard to show that an intersecting subset $S$ of a transitive group $G\le\Sym(\Omega)$ can be much larger than $G_\o$,
	for instance, in a transitive permutation representation of $G=\PSL(2,2^e)$ of degree $|G|/2$, each Sylow 2-subgroup is an intersecting subset.
	This shows that not only a transitive permutation group is in general ``far from satisfying the EKR property'', but also the ratio $|S|/|G_\o|$ can be arbitrarily large.
	
	%
	%
	We note that $|S|<|G|=|G_\o||\Ome|$. Anecdotal evidence indicates that $O\left(|G_\o|\sqrt{|\Ome|}\right)$ is a reasonable estimate the size of an intersecting subset $S$ in primitive groups and quasiprimitive groups, (see Theorem~\ref{qp-EKR} and Lemma~\ref{lem:simp} for example). This conjectured upper bound leads us to define the following invariant of a permutation group.
	
	\begin{definition}\label{ratio}
		{\rm
			Given a transitive permutation group $G\le\Sym(\Ome)$, define
			\[\begin{array}{rcl}
				\rho(G/\Ome)&=&\mbox{${|S|\over|G_\o|\sqrt{|\Ome|}}$,}\\
			\end{array}\]
			where $S$ is a maximum intersecting subset.
		}
	\end{definition}
	
	In the previous unpublished version \cite{LSP}, the notation $\rho(G,\Ome)$ was defined as the maximum value of the ratio ${|S|/|G_\o|}$, which is called \textit{intersection density} and written as $\rho(G)$ in \cite{MRS}.
	So $\rho(G/\Ome)=\rho(G,\Ome)/\sqrt{|\Ome|}$.
	Since $G_{\o}$ is an intersecting subset, we have that $\rho(G/\Ome)\geqslant {1\over\sqrt{|\Ome|}}$, and the equality holds if and only if $G$ satisfies the EKR property.

	A subgroup $R$ of $G\le\Sym(\Ome)$ is {\em semiregular} if the identity is the only element of $R$ which fixes a point of $\Ome$, and then a {\em regular} subgroup is semiregular and transitive.
	These can be extended to subsets.
	
	\begin{definition}\label{def:semireg-sset}
		{\rm
			Let $R$ be a subset of $G$ which contains the identity.
			Then $R$ is called a {\em semiregular subset} if, for any two elements $r,s$, the ratio $rs^{-1}$ fixes no point of $\Ome$.
			A semiregular subset is called a {\it regular subset} if it has size $|\Omega|.$
		}
	\end{definition}
	
	The concept of semiregular subset is important in the study of intersecting subsets, since `semiregular subsets' and `intersecting subsets' are complements to each other.
	
	\begin{proposition}\label{thm:semireg}
		Let $G\le\Sym(\Ome)$ be transitive.
		Let $R$ be a semiregular subset, and $S$ a maximum intersecting subset.
		Then 
		\begin{enumerate}[{\rm(1)}]
			\item $|R||S|\le{|G|}$, and $G=RS$ if $|R||S|=|G|$;
			\item $\rho(G/\Omega)\le \sqrt{|\Omega|}/|R|$;
			\item in particular, if $R$ is a regular subset then $G$ has the EKR property.
		\end{enumerate}
	\end{proposition}

	We remark that \cite[Corollary~2.2]{AM2015} states that if $G$ has a regular subset, then $G$ has the EKR property, and Proposition~\ref{thm:semireg} is a generalisation in some sense. 
	
	A semiregular subset $R$ provides an upper bound for intersecting subsets $|S|\le{|G|\over|R|}$, and if the equality holds, then $S$ is a maximum intersecting subset.
	This motivates us to propose the following `factorization' problem.
	
	\begin{problem}\label{pro:semireg}
		{\rm
			Characterise transitive permutation groups $G\le\Sym(\Ome)$ with a semiregular subset $R$ and an intersecting subset $S$ such that $G=RS$.
		}
	\end{problem}
	
	Next, we consider intersecting subsets in primitive groups and quasiprimitive groups.
	The well-known O'Nan-Scott-Praeger Theorem divides quasiprimitive permutation groups into eight classes, see \cite{DMPbook}.
	The following theorem shows that the EKR property is indeed satisfied for six out of the eight types of quasiprimitive groups; however, a bit surprisingly, the EKR property is ``far'' from being satisfied for some primitive actions of the other two types. 
	
	\begin{theorem}\label{qp-EKR}
		Let $G$ be a quasiprimitive permutation group on $\Ome$.
		Then either
		\begin{itemize}
			\item[(i)] $G$ has a regular subgroup and the EKR property, or
			\item[(ii)] $G$ is an almost simple group or a group of product action type.
		\end{itemize}
		Moreover, the following hold in case {\rm(ii)}.
		\begin{enumerate}[{\rm(a)}]
			\item There exists infinitely many primitive simple groups $G$ such that $G=RS$ for a  semiregular subgroup $R$ and a maximum intersecting subset $S$, and  ${\sqrt2\over2}-\varepsilon<\rho(G/\Ome)<{\sqrt2\over2}$ for any $\varepsilon>0$.
			
			\item There exists infinitely many primitive groups $G$ in product action on $\Omega=\Delta^\ell$ such that $G=RS$ where $R$ is a  semiregular subgroup and $S$ is a maximum intersecting subset, and $\left({\sqrt2\over2}\right)^\ell-\varepsilon<\rho(G/\Ome)<\left({\sqrt2\over2}\right)^\ell$ for any $\varepsilon>0$.
		\end{enumerate}
	\end{theorem}
	
	This theorem and some other evidence motivates us to propose the following conjecture.
	
	\begin{conjecture}\label{problem-1}
		{\rm
			There is a constant $C$ such that $\rho(G/\Ome)\le C$ for finite quasiprimitive permutation groups $G\le\Sym(\Ome)$.
			Is it true that $\lim_{|G|\to\infty}\rho(G/\Ome)\le{\sqrt2\over2}$?
		}
	\end{conjecture}
	
	The following theorem confirms this conjecture for Suzuki groups.	
	
	\begin{theorem}\label{thm:Sz-EKR}
		Let $G=\Sz(q),$ where $q=2^e$ with $e$ odd.
		Then $\rho(G/\Omega)<\frac{\sqrt{2}}{2}$ for each transitive action of $G$ on $\Omega.$
	\end{theorem}
	
	Here is some more evidence for the conjecture $\rho(G/\Ome)<{\sqrt2\over2}$.
	By Proposition~\ref{thm:semireg}, for a semiregular subset $R$ and an intersecting subset, $\rho(G/\Omega)\leqslant \sqrt{|\Omega|}/|R|$.
	Thus, if $G$ has a `large' semiregular subset $R$, say $|R|>\sqrt{2|\Omega|}$, then $\rho(G/\Omega)<{\sqrt2\over2}$.
	Refer to Lemma~\ref{lem:|R|>sqrt|Omega|} and Examples~\ref{ex:PSL(2,p)-big-R}-\ref{ex:PSL-PSp}.
	
	One would hope that Conjecture~\ref{problem-1} would hold for all transitive permutation groups, see \cite[Conjecture\,1.2]{LSP} of the previous version, which conjectured $\rho(G/\Ome)<1$.
	Unfortunately, it is not true.
	Four sporadic counterexamples with $\rho(G/\Ome)>1$ were first discovered in \cite[Theorem\,5.1]{MRS}, via an exhaustive search through the database of all permutation groups of degree less than 48.
	The final theorem of this paper presents infinity families of examples which shows that $\rho(G/\Ome)$ can be arbitrarily large, which solves Problem 1.4 of \cite{LSP}.
	
	\begin{theorem}\label{no-bound}
		For any given integer $M$, there exist infinitely many transitive permutation groups $G$ on $\Omega$ such that
		$\rho(G/\Ome)>M$.
	\end{theorem}
	
	This paper is organized as follows. After this introduction section, we give a group-theoretical proof of Proposition~\ref{thm:semireg}, and apply it to produce some examples of groups with EKR property in Section~\ref{FR}.
	Then we prove Theorem~\ref{qp-EKR} in Section~\ref{large}, and Theorem~\ref{thm:Sz-EKR} in Section~\ref{Qp-EKR}.
	Finally, we construct groups with large intersecting subsets and prove Theorem~\ref{no-bound} in Section~\ref{sec:examples}.

	\section{Semiregular subsets and intersecting subsets}\label{FR}
	
	Let $G\leqslant\Sym(\Ome)$ be transitive, and let $S$ be an intersecting subset of $G$, namely, $xy^{-1}$ fixes some points of $\Ome$ for any elements $x,y\in S$.
	
	For a point $\o\in\Ome$, consider the coset decomposition $G=r_0G_\o \cup r_1G_\o \cup\dots\cup r_{n-1}G_\o$, where $r_0=1$.
	Let
	\[S_i=S\cap(r_iG_\o),\ \mbox{where $0\leqslant i\leqslant n-1$}.\]
	Then $S=S_0\cup S_1\cup\dots\cup S_{n-1}$, and $S_i\cap S_j=\emptyset$.
	For $0\leqslant i\leqslant n-1$, let
	\[H_i=r_i^{-1}S_i=\{r_i^{-1}x\mid x\in S_i\}.\]
	Then $H_i\subset G_\o$, and $|H_i|=|S_i|$, where $0\leqslant i\leqslant n-1$.
	
	\begin{lemma}\label{key-1}
		If $H_i\cap H_j\not=\emptyset$ for some $i,j$ with $i\not=j$, then $r_ir_j^{-1}$ fixes a point.
	\end{lemma}
	\proof
	Let $h\in G_\o$ be such that $h\in H_i\cap H_j$.
	Then $r_ih,r_jh\in S$, and thus
	\[(r_ih)(r_jh)^{-1}=r_ir_j^{-1}.\]
	Since $S$ is an intersecting subset, we have $r_ir_j^{-1}=(r_ih)(r_jh)^{-1}$ fixes some point.
	\qed
	
	We give two proofs for Proposition~\ref{thm:semireg}.
	
	Let $S$ be an intersecting subset, and $R$ a semiregular subset of $G$.
	
	
	

	\subsection*{Proof of Proposition~\ref{thm:semireg}}
	Let $R=\{r_1,r_2,\dots,r_t\}$ with $r_1=1$ be a semiregular subset of $G$, and let $m=|\Ome|/t$.
	Then $R$ has $m$ orbits on $\Ome$, say $\Del_1,\Del_2,\dots,\Del_m$.
	
	Let $\Del_i=\{\d_{i1},\d_{i2},\dots,\d_{it}\}$, where $1\le i\le m$ and $\d_{ij}=\d_{i1}^{r_j}$ with $1\le j\le t$.
	Let $\d_{i1}=\d_{1,1}^{y_i}$, where $y_i\in G$ with $1\le i\le m$.
	Then, $\d_{ij}=\d_{i1}^{r_j}=\d_{11}^{y_ir_j}$, and $\{y_ir_j\mid 1\le i\le m, 1\le j\le t\}$ is a transversal of $G$ on $\Ome$, namely, letting $\o=\d_{11}$,
	\[G=\bigcup_{1\le i\le m}\bigcup_{1\le j\le t}G_\o y_ir_j.\]
	Then we have a coset partition for $S$:
	\[S=S\cap G=\bigcup_{1\le i\le m}\bigcup_{1\le j\le t} \left(S\cap(G_\o y_ir_j)\right)=\bigcup_{1\le i\le m}\bigcup_{1\le j\le t}S_{ij},\]
	where $S_{ij}=S\cap(G_\o y_ir_j)$.
	Since $(G_\o y_ir_j)\cap(G_\o y_{i'}r_{j'})=\emptyset$ for $(i,j)\not=(i',j')$, we have that $S_{ij}\cap S_{i'j'}=\emptyset$, and hence $|S|=\bigcup_{1\le i\le m}\bigcup_{1\le j\le t}|S_{ij}|$.
	Let
	\[\mbox{$H_{ij}=S_{ij}(y_ir_j)^{-1}\subset G_\o$,}\]
	where $1\leqslant i\leqslant m$ and $1\le j\le t$.
	Then $|H_{ij}|=|S_{ij}|$ for all admissible $i,j$, and
	\[|S|=\bigcup_{1\le i\le m}\bigcup_{1\le j\le t}|S_{ij}|=\bigcup_{1\le i\le m}\bigcup_{1\le j\le t}|H_{ij}|.\]
	
	Suppose that $h\in H_{ij}\cap H_{ij'}\not=\emptyset$, where $j\not=j'$.
	Then
	\[\mbox{$h=s_{ij}(y_ir_j)^{-1}$, and $h=s_{ij'}(y_{i}r_{j'})^{-1}$, where $s_{ij}\in S_{ij}$ and $s_{ij'}\in S_{ij'}$.}\]
	Thus $s_{ij}(y_ir_j)^{-1}=s_{ij'}(y_{i}r_{j'})^{-1}$, and so
	\[s_{ij}^{-1}s_{ij'}=(y_ir_j)^{-1}(y_{i}r_{j'})=r_j^{-1}r_{j'}\]
	fixes a point.
	Hence, $r_j^{-1}r_{j'}$ fixes a point, which is not possible since $R$ is a semiregular subset.
	Thus $H_{ij}\cap H_{ij'}=\emptyset$ for $j\not=j'$, and so $\sum_{1\le j\le t}|H_{ij}|\le |G_\o|$.
	Therefore, we conclude that
	\[|S|=\bigcup_{1\le i\le m}\bigcup_{1\le j\le t}|S_{ij}|=\bigcup_{1\le i\le m}\bigcup_{1\le j\le t}|H_{ij}|\le m|G_\o|.\]
	Moreover, $|S|\le m|G_\o|={|\Ome|\over|R|}{|G|\over|\Ome|}={|G|\over|R|}$, and so $|R||S|\le|G|$.
	\qed
	
	
	With Proposition~\ref{thm:semireg}, we construct new examples of permutation groups which have the EKR property.
	
	\begin{lemma}\label{cons:sharply-trnas-set}
		Let $Q\cong\AGL(1,q),$ where $q=p^f$ and $p$ is an odd prime, and suppose $Q$ is a transitive permutation group of degree ${1\over2}q(q-1)$ on a set $\Omega$.
		Then $Q$ contains a sharply transitive set, and thus satisfies the EKR property.
	\end{lemma}
	\begin{proof}
		Let  $P\cong\ZZ_p^f$ be the minimal normal subgroup of $Q$, and $g\in Q$ be of order $q-1$.
		Let $c=g^{q-1\over2}$, the unique involution of $\l g\r$.
		Note that $|Q_{\omega}|=2$ for any $\omega\in \Omega$ and that all involutions in $Q$ are conjugate to $c$ in $Q$. We thus may assume $\Ome=[Q:\l c\r]$.

		If $q\equiv3$ $(\mod 4)$, then ${q-1\over2}$ is odd, and
		$R=P{:}\l g^2\r\cong\ZZ_p^f{:}\ZZ_{q-1\over2}$ is regular on $\Ome$.
		
		Assume that $q=p^f\equiv 1$ $(\mod 4)$.
		Let \[R=Q\setminus \bigcup_{i=1}^{\frac{q-1}{2}}Pg^i.\]
		Then $|R|=q(q-1)-q{q-1\over2}=q{q-1\over2}$.
		We claim that $R$ is sharply transitive on $\Ome$.
		
		Note that each element of $Q$ is of order dividing $q(q-1)$, and $Q_{\omega}=\ZZ_2$.
		Since $R$ does not contain any involution of $Q$, the only element in $R$ that fixes some point in $\Omega$ is the identity.
		Let $x,y$ be distinct elements of $R$.
		Then $x=ag^i$ and $y=bg^j$, where $a,b\in P$ and ${q-1\over2}<i,j\leqslant q-1$.
		Thus
		\[xy^{-1}=(ag^i)(bg^j)^{-1}=ag^{i-j}b^{-1}.g^{-(i-j)}g^{i-j}=ab'g^{i-j},\]
		where $b'=g^{i-j}b^{-1}.g^{-(i-j)}\in P$.
		If $i=j$, then $xy^{-1}=ab'$ is of order $p$ and fixes no point of $\Omega$.
		Suppose that $i\not=j$.
		Since ${q-1\over2}<i,j\leqslant q-1$, it follows that $i-j$ is indivisible by ${q-1\over2}$.
		Thus $xy^{-1}$ is not an involution and fixes no point.
		It follows that $R$ is a sharply transitive set.
	\end{proof}
	
	\begin{corollary}\label{PGL-EKR}
		Let $q=p^f$ with $p$  an odd prime, and $G=\PGL(2,q)$ act on $\Ome$ such that $G_\o=\D_{2(q+1)}$, where $\o\in\Ome$.
		Then $G$ has a sharply transitive subset $R$ and the EKR property.
	\end{corollary}
	\begin{proof}
		Let $Q\leqslant G$ be a maximal parabolic subgroup of $G.$
		Then $Q\cong \AGL(1,q).$
		Note that $G=QG_\omega$ for each $\omega\in \Omega.$
		We thus have $Q\leqslant G$ is transitive of degree $\frac{1}{2}q(q-1).$
		Therefore, $Q\leqslant G$ contains a sharply transitive set of $G$ and $G$ has the EKR property by Lemma \ref{cons:sharply-trnas-set}.
	\end{proof}

		Recall that $\rho(G/\Omega)\leqslant {\sqrt{|\Omega|}\over|R|}$ for semiregular subgroup $R$.
		If $G$ has a `large' semiregular subset, then $\rho(G/\Omega)$ is `small'.
		
		\begin{lemma}\label{lem:|R|>sqrt|Omega|}
			For a transitive permutation group $G$ on $\Omega$, if $G$ has a semiregular subset $R$ such that $|R|\geqslant c\sqrt{|\Omega|}$ for a constant $c$, then $\rho(G/\Omega)\leqslant c^{-1}$.
			
			In particular, if $R<G$ is such that $|R|\geqslant c\sqrt{|G|}$ and $\gcd(|R|,|G_\o|)=1$, then $R$ is a semiregular subgroup and $\rho(G/\Omega)\leqslant c^{-1}$.
		\end{lemma}
		
		\begin{proof}
			The proof follows from $\rho(G/\Omega)\leqslant {\sqrt{|\Omega|}\over|R|}\leqslant c^{-1}$.
		\end{proof}
		
		We end this section with two examples and a lemma which gives more evidence of Conjecture~\ref{problem-1}.
		
		\begin{example}\label{ex:PSL(2,p)-big-R}
			{\rm
				Let $G=\PSL(2,p)$, where $p\ge5$ is a prime, and $G\le\Sym(\Ome)$, with $\o\in\Ome$.
				We consider three cases.
				
				First, assume that $G_\o=\ZZ_p{:}\ZZ_\ell\le\ZZ_p{:}\ZZ_{p-1\over2}$, with $\ell$ an odd integer dividing $\frac{q-1}{2}$.
				Then $|\Omega|=|G:G_\o|=(p^2-1)/2\ell$, and $R\cong\D_{p+1}$ is a semiregular subgroup on $\Omega$, and thus
				\[\rho(G/\Omega)\leqslant {\sqrt{|\Omega|}\over|R|}={\sqrt{(p^2-1)/2\ell}\over p+1}={\sqrt2\over2}\sqrt{p-1\over (p+1)\ell}<{\sqrt2\over2}.\]
				
				Next, assume that $G_\o\leqslant\D_{p+\varepsilon}$ with $\varepsilon=1$ or $-1$.
				Let $x=|G_\o|$.
				Then $|\Omega|={p(p^2-1)\over2x}$ and $R=\ZZ_p$ is semiregular on $\Omega$.
				Thus,
				\[\rho(G/\Omega)\leqslant {\sqrt{|\Omega|}\over|R|}={\sqrt{p(p^2-1)\over 2x}\over p}={\sqrt2\over2}\sqrt{p^2-1\over px}.\]
				In particular, if $G_\o=\D_{p+\varepsilon}$, then $\rho(G/\Omega)\leqslant{\sqrt2\over2}\sqrt{p-\varepsilon\over p}$.
				Therefore,  $\lim\limits_{p\to\infty}\rho(G/\Omega)={\sqrt2\over2}$,
				and $\rho(G/\Omega)<{\sqrt2\over2}$ if further $\epsilon=1.$
				
				\qed
			}
		\end{example}
		
		\begin{example}\label{ex:PSL-PSp}
			{\rm
				Let $G=\PSL(2m,q)$, and $R=\PSp(2m,q)$ with $m$ an integer and $m>1.$
				Let $p$ be a primitive prime divisor of $q^{2m-1}-1$ (which means $p$ divides $q^{2m}-1$ and $p$ does not divide $q^j-1$ for any $j<2m-1$, and $p$ exists by Zsigmondy's Theorem), and let $G<\Sym(\Ome)$ be such that $G_\o$ is a $p$-group.
				Then $\gcd(|R|,|G_\o|)=1$, $R$ is a semiregular subgroup of $G$ on $\Ome$, and $\rho(G/\Ome)<\sqrt{\frac{|G|}{2}}/|R|<\frac{\sqrt{2}}{2}$.
			}
		\end{example}

		\begin{lemma}\label{lem:simgp}
			Let $G$ be a finite group which has a subgroup $H$ and a prime-order element $x$ such that $x^G\cap H= \emptyset.$
			Then $H$ is a semiregular subgroup of $G$ acting on $\Omega=[G:\langle x \rangle].$
			Further, if $|H|^2>|G|,$ then $\rho(G/\Omega)<\frac{1}{\sqrt{2}}.$
		\end{lemma}
	    \begin{proof}
	    	Since each non-identity element in $H$ is not conjugate to any element in $\langle x \rangle$ in $G,$ it follows that each element in $H$ is a derangement of $G$ acting on $\Omega,$ and $H$ is a semiregular subgroup.
	    	If further $|H|^2 > |G|,$ then 
	    	\begin{equation*}
	    		\rho(G/\Omega)\leqslant \frac{\sqrt{|\Omega|}}{|H|}=\frac{\sqrt{|G|/2}}{|H|}<\frac{1}{\sqrt{2}}    	\end{equation*}
    		by Proposition~\ref{thm:semireg}~(2).
    		This completes the proof.
	    	\end{proof}	
    	
    	\begin{lemma}\label{lem:simp}
    		Let $G$ be a non-abelian simple group, then there is a permutation representation of $G$ on $\Omega$ such that $\rho(G/\Omega)<\frac{1}{\sqrt{2}}.$
    	\end{lemma}
        \begin{proof}
            Let $G$ be a non-abelian simple group. 
        	In view of Lemma~\ref{lem:simgp}, it suffices to find a subgroup $H$ of $G$ such that $|H|^2> |G|$ and $G$ contains an prime-order element $x$ with $x^G\cap H=\emptyset.$
        	If $G=M_{11},$ then it is easy to find such subgroup $H$ of $G.$
        	The main result of \cite{M2003} implies that $G$ always contain a prime-order element $x$ with $x^G\cap H=\emptyset$, unless $G=M_{11}$ and $|G:H|=12.$
            We thus may assume $G\neq M_{11},$ and it remains to find a subgroup $H$ of $G$ with $|H|^2> |G|.$ We deal with the cases $G$ is in five family of simple groups, respectively.

        	If $G$ is alternating, then $H=A_{n-1}$ is a subgroup of $G$ with $|H|^2>|G|.$

            If $G$ is a sporadic group, then it is not difficult to find a maximal subgroup $H$ of $G$ with $|H|^2> |G|$ by examining some  maximal groups of $G$ at  \cite{W_WebAt}.
            
            Assume $G$ is a classical group. Then we take $H$ to be a stabilizer of $G$ acting on $1$-dimensional isotropic spaces, and $|H|^2> |G|.$
            
            Finally, we assume $G$ is a simple group of exceptional type.
            Then by a direct calculation, $G$ has a parabolic subgroup $H$ with $|H|^2> |G|.$
            
            This completes the proof.
        \end{proof}

	\section{Proof of Theorem~\ref{qp-EKR}}\label{large}
	
	The proof of Theorem~\ref{qp-EKR} involves constructions of groups without EKR property.
	
	\subsection{Examples from almost simple groups}\
	
	We present examples in the almost simple group case.
	
	\begin{lemma}\label{AS-pri-EKR}
		Let $T=\PSL(2,2^e)$ with $e\geqslant2$, $H=\D_{2(2^e-1)}$, and $\Delta=[T:H]$.
		Then
		\begin{enumerate}[{\rm(1)}]
			\item $T$ is a primitive permutation simple group on $\Delta$ of degree $2^{e-1}(2^e+1)$, and
			\item $T$ has a semiregular subgroup $R\cong \ZZ_{2^e+1}$, and
			\item each intersecting subset of $T$ has size at most $2^e(2^e-1)$, and
			\item $S=\AGL(1,2^e)=\ZZ_2^e{:}\ZZ_{2^e-1}$ is a maximum intersecting subset, and
			\item $ \lim\limits_{e\to \infty}\rho(T/\Del)={\sqrt2\over2}.$
		\end{enumerate}
	\end{lemma}
	\begin{proof}
		Part~(1) is clearly true.
		
		Let $R$ be a cyclic subgroup of $T$ of order $2^e+1$.
		Then $\gcd(|R|,|H|)=1$, and so $R$ is semiregular on $\Del$, as in part~(2).

		By Proposition\ref{thm:semireg}, each intersecting subset of $T$ has size at most $\frac{|G|}{|R|}=2^e(2^e-1)$, as in part~(3).
		
		Let $S=2^e{:}(2^e-1)$ be a maximal parabolic subgroup of $T$.
		Then $S$ is a Frobenius group.
		Let $x$ be an element of $S$.
		Then either $x$ is an involution, or the order $|x|$ divides $2^e-1$.
		The main result of \cite{CJ_cyclic} implies that cyclic subgroups of $T=\PSL(2,2^e)$ of the same order are conjugate.
		Thus, each element of $S$ is conjugate to some element of $H$.
		Thus $S$ is an intersecting subset of $T=\PSL(2,2^e)$ on $\Delta=[T:H]$, as in part~(4).
		
		Finally, we show part (5), we have ${|S|\over|H|}=2^{e-1}$, and $|\Delta|=2^{e-1}(2^e+1)$.
		Thus
		\[\rho(T/\Del)={|S|\over|H|\sqrt{|\Delta|}}={2^{e-1}\over\sqrt{2^{e-1}(2^e+1)}}
		=\sqrt{2^{e-1}\over2^e+1}<{\sqrt2\over2},\]
		and
		\[\lim_{e\to\infty}\rho(T/\Del)=\lim_{e\to\infty}\sqrt{2^{e-1}\over2^e+1}={\sqrt2\over2}.\]
	\end{proof}

	\subsection{Examples from groups of product action type}\
	
	We next introduce a generic construction of intersecting subsets in groups of product action type.	
	
	\begin{construction}\label{prod-cons}
		{\rm
			Let $T\leqslant\Sym(\Delta)$ be transitive with $\delta \in \Del$, and let $P$ be a transitive permutation group of degree $\ell$.
			Let
			\[
			\left\{\begin{array}{l}
				G=T\wr P=T^\ell{:}P,\\
				K=T_\delta\wr P=T_\delta^\ell{:}P,\\
				\Ome=\Delta^\ell.\\
			\end{array}\right.
			\]
			Let $S_0\leqslant T$ be an intersecting subset of $T$, and let $S=S_0^\ell{:}P$.
			\qed
		}
	\end{construction}

	\begin{lemma}\label{prod-lemma}
		Let $G,K,\Ome$ and $S$ be as defined in Construction~$\ref{prod-cons}$.
		Then $S\leqslant G$ is an intersecting subset of $G$.
	\end{lemma}
	\begin{proof}
		By definition, $P$ is a transitive permutation group of degree $\ell$.
		Write elements of $G$ as
		\[\{(t_1,t_2,\dots,t_\ell)\sigma \mid t_1,t_2\dots,t_\ell\in T, \sigma \in P\}.\]
		Pick an element $x\in S$.
		Then there exists a permutation $\sigma_x\in P$ and $\ell$ elements $s_i\in S_0$ for $1\leqslant i \leqslant \ell$ such that
		\[x=(s_1,s_2,\dots, s_\ell)\sigma_x.\]
		For a cycle $\sigma=(n_1,n_2,\dots,n_k)\in \Sy_\ell$ (note that $k\geqslant 1$), we denote the elements of $\{1,2,\dots, \ell\}$ that lies in $\sigma$ by $\mathrm{Supp}(\sigma),$ which means $\mathrm{Supp}(\sigma)=\{n_1,n_2,\dots, n_k \}.$
		Assume that $\sigma_x=\sigma_{x1}\sigma_{x2}\dots \sigma_{xr}$, where $\sigma_{x1},\sigma_{x2},\dots ,\sigma_{xr}$ are $r$ disjoint cycles with
		$$\biguplus_{j=1}^{r}\mathrm{Supp}(\sigma_{xj})=\{1,2,\dots, \ell\}.$$
		Let $L\subseteq \{1,\dots, \ell\}$ such that $|L\cap \mathrm{Supp}(\sigma_{xj})|=1$ for each $1\leqslant j \leqslant r.$
		Then  \cite[Proposition 2.1 ]{CH_conj06} implies there exists $y\in S_0^{\ell}$ such that $x^y=(\Tilde{s_1},\Tilde{s_2},\dots, \Tilde{s_\ell})\sigma_x$, where $\Tilde{s_i}\in S_0$ for $1\leqslant i \leqslant \ell$ and $\Tilde{s_i}=1$ if $i\notin L.$
		Since $S_0\leqslant T$ is an intersecting subset of $T$ , for each $1\leqslant j \leqslant r$ and each $m_j\in L\cap \mathrm{Supp}(\sigma_{xj}),$ $\Tilde{s_{m_j}}=\Tilde{s_{m_j}}e^{-1}$ fixes some point in $\Del$ and there exists $d_j\in T$ such that $\Tilde{s_{m_j}}^{d_j}\in T_\delta.$ Let $z=(d_1',d_2',\dots, d_\ell')\in T^\ell$ such that $d_i'=d_j$ if $i\in \mathrm{Supp}(\sigma_{xj}),$ recall that $\{1,2,\dots, \ell\}=\biguplus_{1\leqslant j \leqslant r}\mathrm{Supp}(\sigma_{xj}).$
		Then $z\sigma_x=\s_xz$, and
		\[x^{yz}=(x^y)^z=(\Tilde{s_1},\Tilde{s_2},\dots, \Tilde{s_\ell})^{z}\sigma_x\in T_\delta^\ell{:}P.\]
		Thus, each $x\in S$ fixes some point in $\Ome,$ and $s_1s_2^{-1}\in S$ fixes some point in $\Omega$ for any $s_1,s_2\in S.$ Therefore, $S$ is an intersecting subset.
	\end{proof}

	\subsection*{Proof of Theorem~\ref{qp-EKR}}\
	
	O'Nan-Scott Theorem (see \cite{Pra1993}) divides primitive permutation groups into eight types, which are categorized into three classes for our purpose:
	\begin{enumerate}[(i)]
		\item $G$ has a subnormal subgroup which is regular, and $G$ is of type HA, HS, HC, SD, CD, or TW;
		
		\item $G$ is an almost simple group;
		
		\item $G$ is in product action.

	\end{enumerate}

	\vskip0.1in
	Let $G\leqslant\Sym(\Ome)$ be quasiprimitive, and let $N=\soc(G)$ be the socle of $G$, namely, the product of all minimal normal subgroups of $G$.
	
	First, assume that $N$ has a normal subgroup which is regular on $\Ome$.
	Then, by Proposition\ref{thm:semireg}, $G$ has the EKR property, as in part~(i) of Theorem~\ref{qp-EKR}.
	
	Assume now that $N$ does not have a normal subgroup which is regular on $\Ome$.
	Then by the O'Nan-Scott theorem, see \cite{Pra1993}, either $G$ is of almost simple type, or $G$ is of product action type, as in part~(ii) of Theorem~\ref{qp-EKR}.
	
	For part~(a) of Theorem~\ref{qp-EKR}\,(ii), the proof then follows from Lemma~\ref{AS-pri-EKR}.
	
	Finally, we need to construct a primitive permutation group of the product action type which has large intersecting subsets.
	
	Let $T=T_i=\PSL(2,2^i)$ and $S_0=\ZZ_2^i{:}\ZZ_{2^i-1}$, where $i\ge2$, and let $G,K,\Ome$ and $S$ be as in Construction~\ref{prod-cons}.
	Then $G$ is a primitive group on $\Ome$ in product action, of degree $(2^{i-1}(2^i+1))^\ell$, $G_\o=K=\D_{2(2^i-1)}^\ell{:}P$.
	Let
	\[S=(\ZZ_2^i{:}\ZZ_{2^i-1})^\ell{:}P.\]
	Then $S$ is an intersecting subset by Lemma~\ref{prod-lemma}.
	
	Let $R_i<T_i$ such that $R_i\cong\ZZ_{2^i+1}$, and let $R=R_1\times\dots\times R_\ell$.
	Then $R$ is semiregular on $\Ome$, with $r=2^{\ell(i-1)}$ orbits.
	Thus an upper bound for an intersecting subset is
	\[r|G_\o|=2^{\ell(i-1)}.(2(2^i-1))^\ell|P|=2^{\ell i}(2^i-1)^\ell|P|.\]
	Thus $S=(\ZZ_2^i{:}\ZZ_{2^i-1})^\ell{:}P$ is a maximum intersecting subset, and
	\[\rho(G/\Ome)={|S|\over|G_\o|\sqrt{|\Ome|}}={2^{\ell(i-1)}\over \sqrt{(2^{i-1}(2^i+1))^\ell}}
	={1\over (2+1/2^{i-1})^{\ell/2}},\]
	and so
	\[\lim_{i\to\infty}\rho(G/\Ome)=\lim_{i\to\infty}{|S|\over|G_\o|\sqrt{|\Ome|}}={1\over 2^{\ell/2}}=\left({\sqrt2\over2}\right)^\ell.\]
	Part~(b) of Theorem~\ref{qp-EKR}\,(ii) then follows, completing the proof of Theorem~\ref{qp-EKR}.
	\qed

	\section{Intersecting subsets of Suzuki groups}\label{Qp-EKR}
	
	In this section, we prove Conjecture~\ref{problem-1} for Suzuki simple groups, and in particular prove Theorem~\ref{thm:Sz-EKR}.

		We first collect some basic properties for Suzuki group $G=\Sz(q).$
		Let $G=\Sz(q)$, where $q=2^e$ with $e\ge3$ odd.
		Then $|G|=q^2(q-1)(q^2+1)$. 
		
		The maximal subgroups of $G$ are recorded in \cite[Theorem 7.3.3]{low-dim-subgps}. By exhausting the subgroups of the maximal subgroups of $G$, it is not difficult to prove the following lemma.
		
		\begin{lemma}
			A solvable subgroup of $G$ is conjuate to a subgroup of one of the following four maximal groups
			\begin{equation*}
				\N_G(S_2)\cong \ZZ_2^e.\ZZ_2^e{:}\ZZ_{q-1}; \;
				\N_G(A_0)\cong \D_{2(q-1)};\; \N_G(A_1)\cong \ZZ_{q+r+1}{:}\ZZ_4;\; \N_G(A_2)\cong \ZZ_{q-r+1}{:}\ZZ_4,
			\end{equation*}
			where $S_2$ is a Sylow $2$-subgroup of $G,$ and $A_0,$ $A_1,$ $A_2$ are three cyclic subgroups of $G$ of order $q-1, q+r+1$ and $q-r+1,$ respectively.
			In particular, for any $x\in G,$ the order of $x$ divides one of $q-1,$ $q+r+1$ and $q-r+1,$ where $r=2^{\frac{e+1}{2}}.$
			
			A non-solvable subgroup of $G$ is isomorphic to $\Sz(q_0),$ where $q=q_0^s$ for some integer $s.$
		\end{lemma} 
		
		\begin{lemma}\label{lem:Sz-basic-pty}
			Let $Q$ be a Sylow $2$-subgroup of $G.$
			Then $\N_G(Q)$ can be identified with $QK,$ where
			\begin{equation*}
				\begin{aligned}	
					Q=\{(\alpha,\beta)\mid \alpha,\beta\in \mathbb{F}_q, (\alpha,\beta)(\alpha_1,\beta_1)&=(\alpha+\alpha_1,\alpha\alpha_1^\theta+\beta+\beta_1)\}\cong \ZZ_2^e.\ZZ_2^e;\\
					K=\{(\kappa)\mid \kappa\in\mathbb{F}_q^\times\}\cong \mathbb{F}_q^\times; & \quad (\alpha,\beta)^\kappa=(\alpha\kappa, \beta\kappa^{1+\theta}),
				\end{aligned}
			\end{equation*}
			and $\theta$ is an automorphism of $\mathbb{F}_q$ such that $\theta^2$ raise every element in $\mathbb{F}_q$ to its square.
			
			In particular, $Q$ contains $q-1$ involutions and $q^2-q$ elements of order $4$, and $Z(Q)$ is an elementary-abelian $2$ group of order $q.$
			Moreover, $Q\cap Q^g=1$ for $g\notin \N_G(Q).$  
		\end{lemma}
		\begin{proof}
		   We refer to \cite[Theorem 9]{Suzuki} and the equation (35) at \cite[Page 133]{Suzuki} for the identification of $\N_G(Q)$ with $QK$.
		   Then by a direct computation, $Z(Q)$ is elementary abelian of order $q,$ and $Z(Q)\backslash 1$ coincides with all involutions in $Q.$
		   Finally, by \cite[Lemma 7.3.2]{low-dim-subgps}, $Q$ intersects trivially with $Q^g$ whenever $g\notin \N_G(Q).$
	   	\end{proof}


		An important combinatorial tool is the derangement graph and its spectrum.
		
		\begin{definition}\label{def:D-graph}
			{\rm
				Let $G$ be a transitive permutation group on $\Omega.$
				Let $D(G,\Omega)$ be the set of elements in $G$ that fix no point in $\Omega,$ and let 
				\[\Gamma_{G,\Omega}=\Cay(G, D(G,\Omega)),\]
				which is called the {\it derangement graph} of $G$ acting on $\Omega.$
			}
		\end{definition}
		
		The importance of a derangement graph is indicated by the following: 
		\begin{itemize}
			\item[(i)] an intersecting set of $G$ acting on $\Omega$ forms a coclique of $\Gamma_{G,\Omega}$, and 
			\item[(ii)] a semiregular subset forms a clique of the graph $\Gamma_{G,\Omega}$.
		\end{itemize}
		
		We now introduce the weighted-adjacency matrix of derangement graph $\Gamma_{G,\Omega}$ and its eigenvalues, and the Hoffman's bound for cocliques in regular graphs.
		Then for each action of $G,$ we use Hoffman bound to limit the size of maximum intersecting subset and prove Theorem~\ref{thm:Sz-EKR}.
		
		For a graph $\Gamma$, a real symmetric matrix $M$ is called \textit{compatible} with $\Gamma$ if its rows and columns are indexed by the vertex of $\Gamma$ such that $M_{u,v}=0$ whenever $u$ is not adjacent to $v$ in $\Gamma$.
		Clearly, the adjacency matrix of $\Gamma$ is compatible with $\Gamma$.
		
		The following lemma states the well-known Delsarte-Hoffman bound, and is used to bound the size of maximum independent set of $\Gamma_{G,\Omega}$.
		
		\begin{lemma}\label{lem:hoff}{\rm(Delsarte-Hoffman bound).}
			Let $M$ be a real symmetric matrix with constant row sum $d$, which is compatible with a graph $\Gamma$ with $n$ vertices.
			If the least eigenvalue of $M$ is $\tau,$ then for any independent set $S$ of $\Gamma$
			\begin{equation*}
				|S|\leqslant \frac{n}{1-\frac{d}{\tau}}.
			\end{equation*}
		\end{lemma}
		
		A class function $f$ from $G$ to $\mathbb{R}$ is called {\it $(G,\Omega)$-compatible} if $f(d)=f(d^{-1})$ for each $d\in G$, and $f(d)=0$ for each $d\notin D(G,\Omega).$
	
		For a compatible class function $f$, let $M^f$ be a matrix of degree $|G|$, which is indexed by the elements in $G$, such that \[M^f_{g_1,g_2}=f(g_1^{-1}g_2).\]
		Then $M^f$ is a symmetric $\Gamma_{G,\Omega}$-compatible matrix, with eigenvalue given in next lemma of Babai, Diaconis and Shahshahani.
		
		\begin{lemma}\label{lem:specm}{\rm(\cite[Lemma 5]{DS_generating})}
			The matrix $M^f$ is a symmetric $\Gamma_{G,\Omega}$-compatible matrix with spectrum
			$$\mathrm{spec}(M^f)=\left\{\frac{1}{\chi(1)}\sum_{g\in G}f(g)\chi(g)\mid \chi\in \mathrm{Irr}_{\mathbb{C}}(G) \right\},$$
			where $\mathrm{Irr}_{\mathbb{C}}(G)$ denotes the set of all irreducible characters of $G$ over $\mathbb{C}.$
		\end{lemma}
		
		We shall denote the eigenvalues of $M^f$ by
		\[\lambda_{\chi, f}=\frac{1}{\chi(1)}\sum_{g\in G}f(g)\chi(g),\ \mbox{where $\chi\in\mathrm{Irr}_{\mathbb{C}}(G)$}.\]
		We notice that the row sum of $M^f$ is $\sum_{g\in G}f(g).$
		
		
		In view of Lemma~\ref{lem:specm}, we note that for any $(G,\Omega)$-compatible class function $f,$ the eigenvalues of $M^f$ is determined by $f$ and the irreducible characters of $G.$
		Thus we collect the character table of $G$ from  \cite[Theorem 13]{Suzuki}, presented in Table~\ref{tab:Szcha}.
		
		Let $r=\sqrt{2q}=2^{e+1\over2}$, and let $A_0, A_1, A_2<G$ be such that
		\[A_0=\l\zeta_0\r\cong \ZZ_{q-1}, \ A_1=\l\zeta_1\r\cong \ZZ_{q+r+1}, \ A_2=\l\zeta_2\r \cong \ZZ_{q-r+1}.\]
		Then each $A_m$ is a Hall subgroup of $G$ for $m\in \{0,1,2\}$.
		Let $\rho$ be an element of order $4$ of $G$.
		Then
		\begin{equation*}
			G=1 \uplus \rho^G \uplus (\rho^{-1})^G \uplus (\rho^2)^G \uplus (A_0\backslash 1)^G \uplus (A_1\backslash 1)^G \uplus (A_2\backslash 1)^G.
		\end{equation*}
		
		\small{
			\begin{table}[h]
				\begin{tabular}{|c|cccccc|}
					\hline
					& $1$ & $\rho^2$ & $\rho$, $\rho^{-1}$ & $\pi_0 \in A_0$ & $\pi_1 \in A_1$ & $\pi_2 \in A_2$ \\
					\hline
					$\bf{1}$ & $1$ & $1$ & $1$ & $1$ & $1$ & $1$ \\
					$X$ & $q^2$ & $0$ & $0$ & $1$ & $-1$ & $-1$\\
					$X_{i}$, $1 \leqslant i\leqslant \frac{q}{2}-1$ & $q^2+1$ & $1$ & $1$ & $\epsilon_0^i(\pi_0)$ & $0$ & $0$\\
					$Y_j, 1\leqslant j \leqslant \frac{q+r}{4}$& $(q-r+1)(q-1)$ & $r-1$ & $-1$ & $0$ & $-\epsilon_1^j(\pi_1)$ & $0$ \\
					$Z_k, 1\leqslant k\leqslant \frac{q-r}{4}$ & $(q+r+1)(q-1)$ & $-r-1$ & $-1$ & $0$ & $0$ & $-\epsilon_2^k(\pi_2)$ \\
					$W_\ell, \ell\in\{1,2\}$ & $\frac{r(q-1)}{2}$ & $-\frac{r}{2}$ & $\pm r\frac{\sqrt{-1}}{2}$ & $0$ & $1$ & $-1$ \\
					\hline		
				\end{tabular}
				\caption{Character table of $\Sz(q)$}
			\end{table}\label{tab:Szcha}
		}
		
		We denote the trivial character of $G$ by $\mathbf{1}.$ As in \cite[Theorem 13]{Suzuki}, we denote the permutation character corresponding to the $2$-transitive action of $G$ by $X,$ and $\{X_i, 1 \leqslant i \leqslant \frac{q}{2}-1\},$ $\{Y_j, 1\leqslant j \leqslant \frac{q+r}{4}\},$ $\{Z_k, 1 \leqslant k \leqslant \frac{q-r}{4}\},$ $\{W_\ell, 1\leqslant \ell \leqslant 2\}$ are four families of irreducible characters of $G$.    
		
		We next explain $\epsilon_0^i,$ $\epsilon_1^j$ and $\epsilon_2^k$ in Table~\ref{tab:Szcha}.
		Recall that $A_m=\l\zeta_m\r$ where $m\in\{0,1,2\}$, and let $e_m$ be a primitive $|A_m|$-th root of unity.
		%
		%
		Then
		\begin{equation*}
			\begin{aligned}
				\epsilon_0^i(\zeta_0^s)&=e_0^{is}+e_0^{-is};\\
				\epsilon_1^j(\zeta_1^s)&=e_1^{js}+e_1^{jsq}+e_1^{-js}+e_1^{-jsq};\\
				\epsilon_2^k(\zeta_2^s)&=e_2^{ks}+e_2^{ksq}+e_2^{-ks}+e_2^{-ksq}.
			\end{aligned}
		\end{equation*}
		
		For each value $m\in\{0,1,2\}$, there exists a cyclic group $B_m\leqslant A_m$ such that 
		\[D(G, \Omega)\cap (A_m\backslash 1)^G=(A_m\backslash B_m)^G.\] 
		We will construct a  $(G,\Omega)$-compatible class function $f$ such that $f$ is constant on $(A_m\backslash B_m)^G$ for $m=0$, $1$ or $2$.
		Then for an irreducible character $\chi$ of $G$, the eigenvalue $\lambda_{\chi, f}$ is calculated as follows.
		
		\begin{lemma}\label{eign}
			Let $f$ be a $(G,\Omega)$-compatible class function which is constant on $(A_m\backslash B_m)^G$ for $m\in\{0,1,2\}$.
			Then, for any irreducible character $\chi$ of $G,$ the eigenvalue of $M^f$ afforded by $\chi$ is equal to
			\begin{equation*}
				\begin{aligned}
					\lambda_{\chi, f}&=\frac{1}{\chi(1)}(f(\rho^2)(q-1)(q^2+1)+f(\rho)q(q-1)(q^2+1)+\\
					&f(\zeta_0)\sum_{g\in (A_0\backslash B_0)^G}\chi(g)+f(\zeta_1)\sum_{g\in (A_1\backslash B_1)^G}\chi(g)+f(\zeta_2)\sum_{g\in (A_2\backslash B_2)^G}\chi(g))
				\end{aligned}
			\end{equation*}
		\end{lemma}
		
		\begin{proof}
			Note that $f(\rho)=f(\rho^{-1})$ since $f$ is a $(G,\Omega)$-compatible class function, it follows that
			\begin{equation}\label{e:1}
				\begin{aligned}
					\lambda_{\chi, f}&=\frac{1}{\chi(1)}(f(\rho^2)|(\rho^2)^G|+f(\rho)|\rho^G\cup (\rho^{-1})^G|\\
					&+f(\zeta_0)\sum_{g\in (A_0\backslash 1)^G}\chi(g)+f(\zeta_1)\sum_{g\in (A_1\backslash 1)^G}\chi(g)+f(\zeta_2)\sum_{g\in (A_2\backslash 1)^G}\chi(g)).\\
				\end{aligned}.
			\end{equation}
			Let $Q$ be a Sylow $2$-subgroup of $G$.
			Then $Q=2^e.2^e$, $Q\cap Q^g=1$ for $g\notin \N_G(Q)$, and $Q$ contains $q-1$ involutions and $q^2-q$ elements of order $4$, see Lemma~\ref{lem:Sz-basic-pty}.
			Then 
			\begin{equation*}
				\begin{aligned}
					|(\rho^2)^G|&=(q-1)\frac{|G|}{|\N_G(Q)|}=(q-1)(q^2+1);\\
					|\rho^G\cup (\rho^{-1})^G|&=(q^2-q)\frac{|G|}{|\N_G(Q)|}=q(q-1)(q^2+1).
				\end{aligned}
			\end{equation*}
		    On the other hand, $\sum\limits_{g\in (B_m)^G}\chi(g)=0$ for $m\in \{0,1,2\},$ this is because $B_m\subseteq G\backslash D(G,\Omega),$ and hence $\sum\limits_{g\in (A_m \backslash 1)^G}\chi(g)=\sum\limits_{g\in (A_m \backslash B_m)^G}\chi(g).$
		
			By substituting the value of $|(\rho^2)^G|$ and $|\rho^G \cup (\rho^{-1})^G|$ in (\ref{e:1}), and replacing $\sum\limits_{g\in (A_m\backslash 1)^G}\chi(g)$ with $\sum\limits_{g\in (A_m\backslash B_m)^G}\chi(g)$ for $m\in \{0,1,2\},$ we complete this proof.
		\end{proof}
		
		By Lemma~\ref{eign}, in order to compute $\lambda_{\chi, f}$, we first compute the summation of $\chi$ on $(A_m\backslash B_m)^G$ for $m=0$, 1, or 2, in the next lemma.
		
		\begin{lemma}\label{e:sumchar}
			Let $B_m\leqslant A_m$ be such that $B_m\cong\ZZ_{t_m}$, where $m\in\{0,1,2\}$.
			Then 
			\begin{equation*}
				\begin{aligned}
					|(A_0\backslash B_0)^G|&=\frac{q^2(q^2+1)(q-1-t_0)}{2};\\
					|(A_1\backslash B_1)^G|&=\frac{q^2(q-1)(q-r+1)(q+r+1-t_1)}{4}; \\
					|(A_2\backslash B_2)^G|&=\frac{q^2(q-1)(q+r+1)(q-r+1-t_2)}{4};
					\\
					\sum_{g\in (A_0\backslash B_0)^G}X_i(g)&=
					\left\{\begin{array}{ll}
						0, &\quad \text{if $t_0$ does not divide $i$,}\\
						-q^2(q^2+1)t_0, &\quad \text{if $t_0$ divides $i$};
					\end{array}
					\right. \\
					\sum_{g\in (A_1\backslash B_1)^G}Y_j(g)&=\left\{\begin{array}{ll}
						0, & \quad \text{if $t_1$ does not divide $j$,} \\
						t_1q^2(q-1)(q-r+1), & \quad \text{if $t_1$ divides $j$;}
					\end{array}
					\right.\\
					\sum_{g\in (A_2\backslash B_2)^G}Z_k(g)&=\left\{\begin{array}{ll}
						0, & \quad \text{if $t_2$ does not divide $k$,} \\
						t_2q^2(q-1)(q+r+1), & \quad \text{if $t_2$ divides $k$.}
					\end{array} \right. \\
				\end{aligned}
			\end{equation*}
		\end{lemma}
		
		\begin{proof}
			Note that if $g\in G$ such that $A_m\cap A_m^g\neq 1$, then $g\in \N_G(A_m\cap A_m^g)=\N_G(A_m)$.
			Thus we have that $A_m\cap A_m^g=1$ for $g\notin \N_G(A_m)$, and so for any irreducible character $\chi$ of $G$,
			$$\sum\limits_{g\in (A_m\backslash B_m)^G}\chi(g)=\frac{|G|}{|\N_G(A_m)|}\left(\sum_{g\in A_m}\chi(g)-\sum_{g\in B_m}\chi(g)\right).$$
			In particular, $|(A_m\backslash B_m)^G|=\sum\limits_{g\in(A_m\backslash B_m)^G}\mathbf{1}(g)=\frac{|G|}{\N_G(A_m)}(|A_m|-|B_m|)$, as stated.
			
			Since $A_0=\langle \zeta_0 \rangle=\ZZ_{t_0}$, we have that $B_0=\langle \zeta_0^{\frac{q-1}{t_0}} \rangle.$
			Thus 
			\[\sum\limits_{g\in B_0}X_i(g)=\left\{
			\begin{array}{ll}
				\mbox{$\sum\limits_{s=1}^{t_0}e_0^{i\frac{q-1}{t_0}s}+\sum\limits_{s=1}^{t_0}e_0^{-i\frac{q-1}{t_0}s}=0$, yielding $\sum\limits_{g\in A_0}X_i(g)=0$,} &\mbox{if $t_0$ does not divide $i$,}\\
				\\
				\sum\limits_{g\in B_0}X_i(g)=2t_0, &\mbox{if $t_0$ divides $i$.}
			\end{array}
			\right.\]
			 Therefore, 
			 $$\sum\limits_{g\in (A_0\backslash B_0)^G}X_i(g)=-\frac{|G|}{|\N_G(A_0)|}\sum_{g\in B_0}X_i(g)=\left\{\begin{array}{ll}
			 	0 & \mbox{if $t_0$ does not divide $i$}\\
			 	-q^2(q^2+1)t_0 & \mbox{if $t_0$ divides $i$}
			 \end{array}  \right. ,$$ as stated in this lemma. 
			
			Arguing similarly shows that the value of $\sum\limits_{g\in (A_1\backslash B_1)^G}Y_j(g)$ and the value of $\sum\limits_{g\in (A_2\backslash B_2)^G}Z_k(g)$ are as stated, completing the proof.
		\end{proof}
		
		 To prove Theorem~\ref{thm:Sz-EKR}, we deal with different actions of $G$ in the following lemmas.

	\subsection{$G_\o\le\D_{2(q-1)}$}\ 
	
	\begin{lemma}\label{lem:d2q-1}
		Assume that $G_\omega\le\D_{2(q-1)}.$
		Then $\rho(G/\Omega)<\frac{\sqrt{2}}{2}.$
		In particular, if $|G_\omega|=2$ then $\rho(G,\Omega)=\frac{q}{2},$ and if $G_\omega \cong \D_{2(q-1)},$ then $\rho(G,\Omega)\leqslant \frac{q^2}{4}.$
	\end{lemma}
	
	\begin{proof}
		Clearly, either $G_\omega \cong \D_{2t_0}$ for some $t_0$ dividing $q-1$, or $G_\omega\cong \ZZ_{t_0}$ for some $t_0$ dividing $q-1$.
		We shall deal with these two cases separately.

		\subsubsection*{Case \rm(i). $G_\omega=\D_{2t_0}$ with $t_0>1$.}\
		
		Let $B_0$ be the cyclic subgroup of order $t_0$ of $A_0.$
		Then
		$$D(G,\Omega)=\rho^G \cup (\rho^{-1})^G \cup (A_0\backslash B_0)^G \cup (A_1\backslash 1)^G \cup (A_2\backslash 1)^G.$$
		Note that $1<t_0<q-1$, or $t_0\in \{1, q-1\}.$ 
		We deal with these three cases separately.
		
		Assume first $1<t_0<q-1.$
		We take $f$ to be the $(G,\Omega)$-compatible class function such that $f(g)=1$ for $g\in D(G,\Omega)$, and $f(g)=0$ for $g\notin D(G,\Omega),$ then corresponding matrix $M^f$ is the adjacency matrix of $\Gamma_{G,\Omega}$.
		By Lemma~\ref{eign} and Lemma~\ref{e:sumchar}, the spectrum of $M^f$ is 
		\begin{equation*}
			\begin{aligned}
				&\lambda_{\mathbf{1},f}=q^5-\frac{1}{2}q^4t_0-\frac{1}{2}q^4-\frac{1}{2}q^2t_0+\frac{1}{2}q^2-q,\\   &\lambda_{X,f}=-\frac{1}{2}q^2t_0+\frac{1}{2}q^2-\frac{t_0}{2}-\frac{1}{2},\\
				&\lambda_{X_i,f}=\left\{\begin{array}{ll}
					q^2-q & \mbox{if $t_0$ does not divide $i$}\\
					-q^2t_0+q^2-q & \mbox{if $t_0$ divides $i$}
				\end{array}\right.,   \\
			    & \lambda_{Y_j,f}= -qr-q,\\
				& \lambda_{Z_k,f}=qr-q,  \\
			    & \lambda_{W_\ell,f}=q^2,
			\end{aligned}
		\end{equation*}
		where $1\leqslant i\leqslant \frac{q^2}{2}, 1\leqslant j \leqslant \frac{q+r}{4},$ $1\leqslant k \leqslant \frac{q-r}{4},$ and $\ell \in \{1,2\}.$
		
		In particular, the row sum of $M^f$ is $d=|D(G,\Omega)|=q^5-\frac{1}{2}q^4t_0-\frac{1}{2}q^4-\frac{1}{2}q^2t_0+\frac{q^2}{2}-q,$ and the minimal eigenvalue $\tau$ of $M^f$ is
		$-q^2t_0+q^2-q.$
		Thus, $1-\frac{d}{\tau}=\frac{q^4-\frac{q^3}{2}(1+t_0)+\frac{q}{2}(t_0-1)}{q(t_0-1)+1},$ and hence
		\begin{equation*}
			\begin{aligned}
				&\frac{1}{\sqrt{2}}\rho(G/\Omega)=\text{max}\{\frac{|S|}{|G_\omega|}/\sqrt{\frac{|\Omega|}{2}}\}=\text{max}\{|S|\sqrt{\frac{2}{|G||G_\omega|}}\}\leqslant \sqrt{\frac{2|G|}{|G_\omega|}}/(1-\frac{d}{\tau})\\
				&=\sqrt{\frac{q^2(q-1)(q^2+1)}{t_0}}\frac{q(t_0-1)+1}{q^4-\frac{q^3}{2}(1+t_0)+\frac{q}{2}(t_0-1)}\\
				&<\frac{q^{\frac{5}{2}}(q(t_0-1)+1)}{\sqrt{t_0}(q^4-\frac{q^3}{2}(1+t_0)+\frac{q}{2}(t_0-1))}\\
				&<\frac{q(t_0-1)+1}{\sqrt{qt_0}(q-\frac{t_0}{2}-1)},
			\end{aligned}
		\end{equation*}
		where the first inequality is by Lemma~\ref{lem:hoff}.
		Since $t_0\geqslant 7$ ($t_0\geqslant 7$ because both $3$ and $5$ are not divisors of $q-1$) and $1<t_0<q-1$, it follows that $\frac{1}{\sqrt{2}}\rho(G/\Omega)<\frac{t_0(q-1)}{\sqrt{7t_0^2}\frac{13}{14}(q-1)}<1.$
		This completes the case $t_0$ divides $q-1$ and $1<t_0<q-1.$
		
		Assume next $t_0=q-1$ and $G_\omega\cong \D_{2(q-1)}.$
		Then
		\begin{equation*}
			D(G,\Omega)=\rho^G\cup (\rho^{-1})^G \cup (A_1\backslash 1)^G \cup (A_2\backslash 1)^G.
		\end{equation*}
		We take $f$ to be a $(G,\Omega)$-compatible class function such that $f(g)=1$ for $g\in \rho^G\cup (\rho^{-1})^G,$ $f(g)=\frac{2}{r}+\frac{2}{q}$ for $g\in (A_1\backslash 1)^G,$ $f(g)=-\frac{2}{r}+\frac{2}{q}$ for $g\in (A_2\backslash 1)^G$, and $f(g)=0$ for $g\notin D(G,\Omega).$
		Then by Lemma~\ref{eign} and \ref{e:sumchar}, the spectrum of $M^f$ is
		\begin{equation*}
			\begin{aligned}
				& \lambda_{\mathbf{1},f}=2q^4-2q^3+q^2-q, \\
				& \lambda_{X,f}=\lambda_{Y_j,f}=\lambda_{Z_k,f}=-q^2+q,\\ 
				& \lambda_{X_i,f}=q^2-q, \\
				& \lambda_{W_\ell,f}=q^3-q^2+2q, \\
			\end{aligned}
		\end{equation*}
		where $1\leqslant i\leqslant \frac{q^2}{2}, 1\leqslant j \leqslant \frac{q+r}{4},$ $1\leqslant k \leqslant \frac{q-r}{4},$ and $\ell \in \{1,2\}.$
		In particular, the row sum of $M^f$ is $d=2q^4-2q^3+q^2-q,$ and the minimal eigenvalue of $M^f$ is $\tau=-q^2+q.$
		Thus, for any intersecting subset $S$ of $G$,
		\begin{equation*}
			|S|\leqslant \frac{|G|}{1-\frac{d}{\tau}}=\frac{q^2(q-1)(q^2+1)}{1+\frac{2q^4-2q^3+q^2-q}{q(q-1)}}=\frac{q^2(q-1)}{2}
		\end{equation*}
		by Lemma~\ref{lem:hoff}.
		It follows that
		\begin{equation*}
			\rho(G,\Omega)=\frac{|S|}{|G_\omega|}\leqslant \frac{q^2}{4}<\frac{q\sqrt{q^2+1}}{2}=\sqrt{\frac{|\Omega|}{2}},
		\end{equation*}
		and hence $\rho(G/\Omega)<\frac{\sqrt{2}}{2}.$
		This completes the case $t_0=q-1$ and $G_\omega=\D_{2(q-1)}$.
		
		Finally, we assume $t_0=1$ and $G_\omega\cong \ZZ_2.$
		It is clear that the center of any Sylow $2$-subgroup of $G$ is an intersecting subset (see Lemma~\ref{lem:Sz-basic-pty} for the structure of Sylow $2$-subgroup of $G$).
		On the other hand, for any intersecting subset $S\subseteq G$ containing $1,$ the subgroup generated by $S$ is elementary-abelian by \cite[Lemma 3.1]{HKKM_cycstab}.
		Thus, $S$ is contained in the center of some Sylow $2$-subgroup of $G$.
		Therefore, the maximum intersecting subsets of $G$ containing $1$ are the centers of some Sylow $2$-subgroup, and hence $\rho(G,\Omega)=\frac{q}{2}.$
		
		This completes the case $G_\omega=\D_{2t_0}$.
		
		\subsubsection*{Case \rm(ii): $G_\omega\cong \ZZ_{t_0}$}
		Let $B_0$ be a subgroup of $A_0$ of order $t_0.$
		Then it is clear that
		\begin{equation*}
			D(G,\Omega)=(\rho^2)^G\cup \rho^G \cup (\rho^{-1})^G\cup (A_0\backslash B_0)^G \cup (A_1\backslash 1)^G \cup (A_2\backslash 1)^G.
		\end{equation*}
		Similarly as Case \rm(i), we deal with the cases $1<t_0<q-1$ and $t_0=q-1$ separately.
		
		Assume first $1<t_0<q-1$.
		We take $f$ to be the $(G,\Omega)$-compatible function such that $f(g)=1$ for $g\in D(G,\Omega),$ and $f(g)=0$ for $g\notin D(G,\Omega),$  then $M^f$ is the adjacency matrix of the derangement graph of this action.
		By Lemma~\ref{eign} and Lemma~\ref{e:sumchar}, the spectrum of $M^f$ is
		\begin{equation*}
			\begin{aligned}
				&\lambda_{\mathbf{1},f}=q^5-\frac{q^4}{2}t_0-\frac{q^4}{2}+q^3-\frac{q^2}{2}t_0-\frac{q^2}{2}-1,\\
				&\lambda_{X,f}=-\frac{q^2}{2}t_0+\frac{q^2}{2}-\frac{t_0}{2}-\frac{1}{2},\\	
				&\lambda_{X_i,f}=\left\{\begin{array}{ll}
					q^2-1 & \text{if $t_0$ does not divide $i$} \\
					-q^2t_0+q^2-1 & \text{if $t_0$ divides $i$}
				\end{array}
				\right. , \\
				&\lambda_{Y_j,f}=\lambda_{Z_k,f}=\lambda_{W_\ell,f}=-1,
			\end{aligned}\\
		\end{equation*}
		where $1\leqslant i\leqslant \frac{q^2}{2}, 1\leqslant j \leqslant \frac{q+r}{4},$ $1\leqslant k \leqslant \frac{q-r}{4},$ and $\ell \in \{1,2\}.$
		In particular, the row sum of $M^f$ is $d=|D(G, \Omega)|=q^5-\frac{q^4}{2}t_0-\frac{q^4}{2}+q^3-\frac{q^2}{2}t_0-\frac{q^2}{2}-1,$ and the minimal eigenvalue of $M^f$ is $\tau=-q^2t_0+q^2-1.$
		Thus,
		\begin{equation*}
			\begin{aligned}
				1-\frac{d}{\tau}&=1+\frac{q^4(q-\frac{t_0+1}{2})+q^2(q-\frac{t_0+1}{2})-1}{q^2(t_0-1)+1}=1+\frac{q^2(q^2+1)(q-\frac{t_0+1}{2})-1}{q^2(t_0-1)+1}\\
				&=\frac{q^2(q^2+1)(q-\frac{t_0+1}{2})+q^2(t_0-1)}{q^2t_0-q^2+1}> \frac{q^2(q^2+1)(q-\frac{t_0+1}{2})}{q^2t_0}=\frac{(q^2+1)(q-\frac{t_0+1}{2})}{t_0}.
			\end{aligned}
		\end{equation*}
		Since $|S|\leqslant \frac{|G|}{1-\frac{d}{\tau}}$ by Lemma~\ref{lem:hoff},
		it follows that
		\begin{equation*}
			\begin{aligned}
				\frac{1}{\sqrt{2}}\rho(G/\Omega)&=\text{max}\{\frac{|S|}{|G_\omega|}\sqrt{\frac{|\Omega|}{2}}\}=\text{max}\{|S|\sqrt{\frac{2}{|G||G_\omega|}} \}\leqslant \sqrt{\frac{2|G|}{|G_\omega|}}/(1-\frac{d}{\tau})\\
				&<\sqrt{\frac{2q^2(q-1)(q^2+1)}{t_0}}\frac{t_0}{(q^2+1)(q-\frac{t_0+1}{2})}\\
				&<\frac{q^2\sqrt{2qt_0}}{(q^2+1)(q-\frac{t_0+1}{2})}.
			\end{aligned}
		\end{equation*}
		On the other hand, $t_0\leqslant \frac{q-1}{7}$ since $3$ and $5$ are not divisor of $q-1$ and $1<t_0<q-1.$
		Thus, $$\frac{1}{\sqrt{2}}\rho(G/\Omega)<\frac{q^2}{q^2+1}\frac{\sqrt{2q(q-1)/7}}{\frac{13}{14}q-\frac{3}{7}}<\frac{\sqrt{2/7}}{6/7}<1,$$ and hence $\rho(G/\Omega)<\frac{\sqrt{2}}{2}.$
		
		We next assume $t_0=q-1$ and $G_\omega\cong \ZZ_{q-1}.$
		Then
		\begin{equation*}
			D(G,\Omega)=(\rho^2)^G \cup \rho^G \cup (\rho^{-1})^G \cup (A_1\backslash 1)^G \cup (A_2\backslash 1)^G.
		\end{equation*}
		We take the $(G,\Omega)$-compatible function $f$ as follows: $f(g)=1$ for $g\in \rho^2\cup \rho^G \cup (\rho^{-1})^G$,  $f(g)= \frac{2}{q}$ for $g\in (A_1\backslash 1)^G \cup (A_2\backslash 1)^G$, and $f(g)=0$ for $g\notin D(G,\Omega).$
		By Lemma~\ref{eign} and Lemma~\ref{lem:specm}, the spectrum of the $\Gamma_{G,\Omega}$-compatible matrix $M^f$ is
		\begin{equation*}
			\begin{aligned}
				&\lambda_{\mathbf{1},f}=2q^4-2q^3+q^2-1, \\ &\lambda_{X_i}=q^2-1, \\
				&\lambda_{X}=\lambda_{Y_j}=\lambda_{Z_k}=\lambda_{W_\ell}=-(q-1)^2,
			\end{aligned}
		\end{equation*}
		where $1\leqslant i \leqslant \frac{q}{2}-1, 1\leqslant j \leqslant \frac{q+r}{4}, 1\leqslant k \leqslant \frac{q-r}{4},$ and $\ell\in \{1,2\}.$
		In particular, the row sum of $M^f$ is $d=2q^4-2q^3+q^2-1$ and the minimal eigenvalue of $M^f$ is $\tau=-(q-1)^2.$
		Thus, for any intersecting subset $S\subseteq G,$
		\begin{equation*}
			|S|\leqslant \frac{|G|}{1-\frac{d}{\tau}}=
			\frac{q^2(q-1)(q^2+1)}{1+\frac{2q^4-2q^3+q^2-1}{(q-1)^2}}=\frac{q(q-1)^2}{2},
		\end{equation*}
		by Hoffman's bound (Lemma~\ref{lem:hoff}).
		Therefore, $$\rho(G,\Omega)=\frac{|S|}{|G_\omega|}\leqslant \frac{q(q-1)}{2}<\frac{q\sqrt{q^2+1}}{\sqrt{2}}=\sqrt{\frac{|\Omega|}{2}},$$ and hence $\rho(G/\Omega)<\frac{\sqrt{2}}{2}$
		This completes the case $G_\omega=\ZZ_{q-1}.$
		
		This completes the proof.
	\end{proof}
	
	\subsection{$G_\o\leqslant \ZZ_2^e.\ZZ_2^e{:}\ZZ_{q-1}$}\ 
	
	\begin{lemma}\label{subBorel}
		Assume $G_\omega$ normalizes some Sylow $2$-subgroup of $G.$
		Then $\rho(G/\Omega)< \frac{\sqrt{2}}{2}.$
		In particular, the following holds
		\begin{itemize}
			\item [\rm(i)] \cite[Proposition 4.1]{MST2016} If $G_\omega$ is a maximal subgroup of $G,$ then $\rho(G,\Omega)=1;$
			\item [\rm(ii)] If $G_\omega$ is a $2$-group with exponent $4$, then  $\rho(G,\Omega)=\frac{q^2}{|G_\omega|}$;
			\item [\rm(iii)] If $G_\omega$ is an elementary-abelian $2$-group, then $\rho(G,\Omega)=\frac{q}{|G_\omega|}$;
			\item [\rm(iv)] If $G_\omega\cong \ZZ_2^e{:}\ZZ_{q-1},$ then $\rho(G,\Omega)\leqslant \frac{q}{2}$;
		\end{itemize}
		
	\end{lemma}

	\begin{proof}
		We divide the proof into two cases: $G_\omega$ contains an element of order $4$ and $G_\omega$ does not contain any element of order $4,$ and we deal with them separately.
		
		\subsubsection*{Case \rm(i): $G_\omega$ contains an element of order $4$}
		
		Assume $G_\omega$ has a Hall $2'$-subgroup of order $t_0.$
		We deal with the cases $t_0=1$ and $t_0>1,$ respectively.
		
		First assume $t_0>1.$
		Let $\lambda$ be an element in $\mathbb{F}_q^\times$ of order $t_0$, and assume $\mathbb{F}_2[\lambda]=\mathbb{F}_{q_0}.$
		We next claim that $|G_\omega|\geqslant q_0^2t_0.$
		To see this, recall that the maximal subgroup of $G$ containing $G_\omega$ can be identified with $QK,$ where 
		\begin{equation*}
			\begin{aligned}	
				Q=\{(\alpha,\beta)\mid \alpha,\beta\in \mathbb{F}_q, (\alpha,\beta)(\alpha_1,\beta_1)&=(\alpha+\alpha_1,\alpha\alpha_1^\theta+\beta+\beta_1)\}\cong \ZZ_2^e.\ZZ_2^e;\\
				K=\{(\kappa)\mid \kappa\in\mathbb{F}_q^\times\}\cong \mathbb{F}_q^\times; & \quad (\alpha,\beta)^\kappa=(\alpha\kappa, \beta\kappa^{1+\theta}),
			\end{aligned}
		\end{equation*}
	    by Lemma~\ref{lem:Sz-basic-pty}.
		We may further replace $G_\omega$ with $G_\omega^{g_1}$ for some $g_1$ in the maximal subgroup containing $G_\omega,$ and assume $(1,1)\in G_\omega,$ and hence $(0,1)=(1,1)^2\in G_\omega.$
		Note that $G_\omega\cap Z(Q)\cong \{\beta\in \mathbb{F}_q\mid (0,\beta)\in G_\omega\},$ and $$(G_\omega\cap Q)/(G_\omega\cap Z(Q))\cong \{\alpha \mid \text{$(\alpha,\beta')\in G_\omega$ for some $\beta'\in \mathbb{F}_q$}\}.$$
		Also note that there exists $\alpha_1,\beta_1 \in \mathbb{F}_q,$ such that $((\alpha_1,\beta_1),\lambda)$ generates a Hall $2'$-subgroup of $G_\omega,$ and
		$$(\alpha,\beta)^{((\alpha_1,\beta_1),\lambda)}=(\lambda\alpha,\lambda^{1+\theta}(\beta+\alpha_1\alpha^\theta+\alpha\alpha_1^\theta)(\alpha_1^\theta+1)), \quad  (0,\beta)^{((\alpha_1,\beta_1),\lambda)}=(0,\beta\lambda^{1+\theta}).$$        	
		It follows that $G_\omega\cap Z(Q)$ and $(G_\omega\cap Q)/(G_\omega\cap Z(Q))$ can be identified two additive subgroups $E,E_1$ of $\mathbb{F}_q,$ such that both $E$ and $E_1$ contain $1,$ and they are closed under the multiplication of $\lambda.$
		Thus, $\mathbb{F}_2[\lambda]=\mathbb{F}_{q_0}\subseteq E\cap E_1,$ and hence $G_\omega\geqslant q_0^2.$ This completes the claim.
		
		We next assign a $(G,\Omega)$-compatible function $f,$ and use the minimal eigenvalue of $M^f$ and the Hoffman-bound to prove $\rho(G/\Omega)<\frac{\sqrt{2}}{2}.$
		Let $B_0$ be the unique  subgroup of $A_0$ of order $t_0.$
		Then $$D(G,\Omega)=(A_0\backslash B_0)^G \cup (A_1\backslash 1)^G \cup (A_2\backslash 1)^G.$$
		We take $f$ such that $f(g)=1$ for $g\in D(G,\Omega)$ and $f(g)=0$ for $g\notin D(G,\Omega)$, then $M^f$ is the adjacency matrix of $\Gamma_{G,\Omega}.$
		By Lemma~\ref{eign} and Lemma~\ref{lem:hoff}, the spectrum of $M^f$ is
		\begin{equation*}
			\begin{aligned}
				&\lambda_{\mathbf{1},f}=q^5-\frac{q^4}{2}t_0-\frac{3}{2}q^4+q^3-\frac{q^2}{2}t_0-\frac{q^2}{2}, \\
				& \lambda_{X,f}=-\frac{q^2}{2}t_0+\frac{q^2}{2}-\frac{t_0}{2}-\frac{1}{2},\\
				&\lambda_{X_i,f}=\left\{ \begin{array}{ll}
					0 & \quad \mbox{if $t_0$ does not divide $i$}\\
					-q^2t_0 & \quad \mbox{if $t_0$ divides $i$}
				\end{array}\right., \\
			    & \lambda_{Y_j,f}=\lambda_{Z_k,f}=\lambda_{W_\ell,f}=q^2, \\
			\end{aligned}
		\end{equation*}
		where $1\leqslant i\leqslant \frac{q^2}{2}, 1\leqslant j \leqslant \frac{q+r}{4},$ $1\leqslant k \leqslant \frac{q-r}{4},$ and $\ell \in \{1,2\}.$
		In particular, the row sum of $M^f$ is $d=|D(G,\Omega)|=q^5-\frac{q^4}{2}t_0-\frac{3}{2}q^4+q^3-\frac{q^2}{2}t_0-\frac{q^2}{2},$ and the minimal eigenvalue of $M^f$ is $\tau=-q^2t_0.$
		Also note that $\frac{q}{t_0}> \frac{q}{q_0}=q_0^{r-1},$
		and hence  $\frac{(q-\frac{t_0}{2}-\frac{1}{2})}{t_0}>\frac{(q-\frac{t_0}{2}-\frac{3}{2})}{t_0}>q_0^{r-1}-1.$
		It follows that
		\begin{equation*}
			\begin{aligned}
				1-\frac{d}{\tau}&=1+\frac{q^4(q-\frac{t_0}{2}-\frac{3}{2})+q^2(q-\frac{t_0}{2}-\frac{1}{2})}{q^2t_0}\\> &1+q^2(q_0^{r-1}-1)+(q_0^{r-1}-1)>(q^2+1)(q_0^{r-1}-1).
			\end{aligned}
		\end{equation*}
		Thus, by Hoffman bound (Lemma~\ref{lem:hoff}), for any intersecting subset $S,$
		\begin{equation*}
			|S|\leqslant \frac{|G|}{1-\frac{d}{\tau}}<\frac{q^2(q-1)(q^2+1)}{(q^2+1)(q_0^{r-1}-1)}=\frac{q^2(q-1)}{q_0^{r-1}-1}.
		\end{equation*}
		On the other hand, $|G_\omega|\geqslant q_0^2t_0$ by the second paragraph.
		It follows that
		\begin{equation*}
			\frac{1}{\sqrt{2}}\rho(G/\Omega)=\text{max}\{|S|\sqrt{\frac{2}{|G||G_\omega|}}\}<\sqrt{\frac{2}{q^2(q-1)(q^2+1)q_0^2t_0}}\frac{q^2(q-1)}{q_0^{r-1}-1}<1,
		\end{equation*}
		and hence $\rho(G/\Omega)<\frac{\sqrt{2}}{2}.$
		This completes the case $t_0>1.$
		
		It remains to consider $t_0=1$ and $G_\omega$ is a $2$-group.
		Then
		\begin{equation*}
			D(G,\Omega)=(A_0\backslash 1)^G\cup (A_1\backslash 1)^G\cup (A_2\backslash 1)^G.
		\end{equation*}
		We take $f$ to be the $(G,\Omega)$-compatible function such that $f(g)=1$ for $g\in (A_0\backslash 1)^G,$ and $f(g)=1+\frac{2(q+1)}{q(q-1)}$ for $g\in (A_1\backslash 1)^G \cup (A_2\backslash 1)^G,$ and $f(g)=0$ for remaining elements in $G.$
		By Lemma~\ref{eign} and Lemma~\ref{lem:hoff}, the spectrum of $M^f$ is
		\begin{equation*}
			\begin{aligned}
			&\lambda_{\bf{1},f}=q^5-q^4+q^3-2q^2, \\ &\lambda_{X,f}=\lambda_{X_i,f}=-q^2,  \\ &\lambda_{Y_j,f}=\lambda_{Z_k,f}=\lambda_{W_\ell,f}= \frac{q^3+q^2+2q}{q-1}, \\
			\end{aligned}
		\end{equation*}
		where $1\leqslant i \leqslant \frac{q}{2}-1, 1\leqslant j \leqslant \frac{q+r}{4}, 1\leqslant k \leqslant \frac{q-r}{4},$ and $\ell\in \{1,2\}.$
		In particular, the row sum of $M^f$ is $d=q^5-q^4+q^3-2q^2$, and the minimal eigenvalue $\tau$ is $-q^2.$ Therefore, for any intersecting subset $S\subseteq G,$ we have
		\begin{equation*}
			|S|\leqslant \frac{|G|}{1-\frac{d}{\tau}}=\frac{q^2(q-1)(q^2+1)}{1+\frac{q^5-q^4+q^3-2q^2}{q^2}}=q^2.
		\end{equation*}
		On the other hand, it is clear that any Sylow $2$-subgroup of $G$ is an intersecting subset with order $q^2.$ Thus, Sylow $2$-subgroups of $G$ are maximum intersecting subsets, and hence \rm(ii) of this lemma holds.
		This completes the case $G_\omega$ contains an element of order $4.$
		
		\subsubsection*{Case \rm(ii)}
		$G_\omega$ does not contain any element of order $4$.
		
		If $|G_\omega|$ is odd, then $G_\omega\cong \ZZ_{t_0},$ and this is just Case~\rm(ii) of Lemma~\ref{lem:d2q-1}.
		We thus may assume $|G_\omega|$ is even, then it is not difficult to see that
		\begin{equation*}
			D(G,\Omega)=\rho^G \cup (\rho^{-1})^G \cup (A_0\backslash B_0)^G \cup (A_1 \backslash 1)^G \cup (A_2 \backslash 1)^G.
		\end{equation*}
		Let $H$ be a subgroup of $G$ isomorphic to $\D_{2t_0}.$
		Then $D(G,\Omega)=D(G, [G:H]),$ and hence any intersecting subset $S$ of $G$ on $\Omega$ is also an intersecting subset of $G$ acting on $[G:H].$
		Therefore,
		\begin{equation*}
			\frac{1}{\sqrt{2}}\rho(G/\Omega)=\text{max}\{\frac{|S|}{|G_\omega|}/\sqrt{\frac{|\Omega|}{2}}\}=\text{max}\{|S|\sqrt{\frac{2}{|G||G_\omega|}}\}<\text{max}\{|S|\sqrt{\frac{2}{|G||H|}}\}<1
		\end{equation*}
		by the proof of Lemma~\ref{lem:d2q-1}.
		For the case when $G_\omega \cong \ZZ_{2^e}{:}\ZZ_{q-1}$ and $G_\omega$ is an elementary abelian $2$-group, we also derive \rm(iii) and \rm(iv) for the statement of this lemma from the proof of Lemma~\ref{lem:d2q-1}.
		This completes the case $G_\omega$ does not contain any element of order $4.$
		
		This completes the proof.
	\end{proof}

	\subsection{Subfield subgroups}\ 
	
	\begin{lemma}\label{szq1}
		If $G_\omega\cong \Sz(q_1),$ where $q=q_1^k$ for some integer $k$ and $q_1> 2$, then
		$|S|< 2q^2q_1$ for any intersecting subset $S\subseteq G,$ and $\rho(G/\Omega)<\frac{\sqrt{2}}{2}.$
	\end{lemma}
	\begin{proof}
		Let $B_0\leqslant A_0, B_1\leqslant A_1, B_2\leqslant A_2$ be cyclic subgroups of $G$ of order $t_0=q_1-1, t_1=q_1+\sqrt{2q_1}+1$ and $t_2=q_1-\sqrt{2q_1}+1$ respectively.
		Then 		
		\begin{equation*}
			D(G,\Omega)=(A_0\backslash B_0)^G \cup (A_1\backslash B_1)^G \cup (A_2\backslash B_2)^G.
		\end{equation*}
		We take the $(G,\Omega)$-compatible class function $f$ such that $f(g)=1$ for $g\in (A_0\backslash B_0)^G,$ and $f(g)=0$ for $g \notin (A_0 \backslash B_0)^G.$
		By Lemma~\ref{eign} and Lemma~\ref{lem:hoff}, the spectrum of $M^f$ is
		\begin{equation*}
			\begin{aligned}
				& \lambda_{\mathbf{1},f}=\frac{q^2(q-q_1)(q^2+1)}{2},\\
				& \lambda_{X,f}=\frac{(q-q_1)(q^2+1)}{2},\\
				&\lambda_{X_i,f}=\left\{\begin{array}{ll}
					0 &  \mbox{if $q_1-1$ does not divide $i$} \\
					-q^2(q_1-1) & \mbox{if $q_1-1$ divides $i$}
				\end{array} \right., \\ &\lambda_{Y_j,f}=\lambda_{Z_k,f}=\lambda_{W_\ell}=0,
			\end{aligned}
		\end{equation*}
		where $1\leqslant i\leqslant \frac{q^2}{2}, 1\leqslant j \leqslant \frac{q+r}{4},$ $1\leqslant k \leqslant \frac{q-r}{4},$ and $\ell \in \{1,2\}.$
		In particular, the row sum of $M^f$ is $d=\frac{q^2(q-q_1)(q^2+1)}{2}$ and the minimal eigenvalue of $M^f$ is $\tau=-q^2(q_1-1).$
	Thus, by Hoffman's bound (Lemma~\ref{lem:hoff}),
	\begin{equation*}
		|S|\leqslant \frac{|G|}{1-\frac{d}{\tau}}=\frac{q^2(q-1)(q^2+1)}{1+\frac{q^2(q-q_1)(q^2+1)}{2q^2(q_1-1)}}< \frac{q^2(q-1)(q^2+1)}{\frac{(q-q_1)(q^2+1)}{2(q_1-1)}}=\frac{2q^2(q-1)(q_1-1)}{q-q_1}<2q^2q_1,
	\end{equation*}
	and hence
	\begin{equation*}
		\begin{aligned}
			\frac{1}{\sqrt{2}}\rho(G/\Omega) &=\text{max}\{\frac{|S|}{|G_\omega|}/\sqrt{\frac{|\Omega|}{2}}\}=\text{max}\{|S|\sqrt{\frac{2}{|G||G_\omega|}}\}\\ &<\sqrt{\frac{2}{q^2q_1^2(q-1)(q_1-1)(q^2+1)(q_1^2+1)}}2q^2q_1\\
			&<\frac{2\sqrt{2}}{\sqrt{(q-1)(q_1-1)(q_1^2+1)}}<1.
		\end{aligned}
	\end{equation*}
	Therefore, $\rho(G/\Omega)<\frac{\sqrt{2}}{2}.$ This completes the proof.
	\end{proof}
	
	\subsection{$G_\o\leqslant\ZZ_{q\pm r+1}{:}\ZZ_4$}\

	\begin{lemma}\label{q+r+1}
		Assume $G_\omega$ is a subgroup of $\ZZ_{q+r+1}{:}\ZZ_4$.
		Then $|S|\leqslant 2q\frac{q^2+1}{q-1}$ for any intersecting subset $S\subseteq G$ and $\rho(G/\Omega)<\frac{\sqrt{2}}{2}$.
	\end{lemma}
	\begin{proof}
		Let $H$ be a maximal subgroup of $G$ isomorphic to $\ZZ_{q+r+1}{:}\ZZ_4$.
		Then 
		$$ D(G, [G:H])=(A_0\backslash 1)^G \cup (A_2\backslash 1)^G \subseteq D(G,\Omega),$$ where $D(G, [G:H])$ is the derangement set of the action of $G$ on $[G:H].$
		Let $f$ be a function from $G$ to $\mathbb{R}$ such that
		$f(g)=1$ for $g\in (A_0\backslash 1)^G$ and $f(g)=0$ for $g\notin (A_0\backslash 1)^G.$
		Then $f$ is a $(G,\Omega)$-compatible class function.
		By Lemma~\ref{eign} and Lemma~\ref{lem:hoff}, the spectrum of $M^f$ is
		\begin{equation*}
			\begin{aligned}
			&\lambda_{\mathbf{1},f}=\frac{q^5}{2}-q^4+\frac{q^3}{2}-q^2,\\ &\lambda_{X,f}=\frac{q^3}{2}-q^2+\frac{q}{2}-1, \\ &\lambda_{X_i,f}=-q^2, \\ &\lambda_{Y_j,f}=\lambda_{Z_k,f}=\lambda_{W_\ell,f}=0,
		\end{aligned}
	    \end{equation*}
		where $1\leqslant i\leqslant \frac{q^2}{2}, 1\leqslant j \leqslant \frac{q+r}{4},$ $1\leqslant k \leqslant \frac{q-r}{4},$ and $\ell \in \{1,2\}.$
		In particular, the row sum of $M^f$ is $d=\frac{q^5}{2}-q^4+\frac{q^3}{2}-q^2,$ and the minimal eigenvalue of $M^f$ is $\tau=-q^2.$
		Then by Hoffman's bound (Lemma~\ref{lem:hoff}), for any intersecting subset $S\subseteq G,$
		\begin{equation*}
			|S|\leqslant \frac{|G|}{1-\frac{d}{\tau}}=\frac{q^2(q-1)(q^2+1)}{\frac{q^3}{2}-q^2+\frac{q}{2}}=2q\frac{q^2+1}{q-1}.
		\end{equation*}
		If $G_\omega$ is a two-group, then $|S|\leqslant q^2$ for any intersecting subset $S\subseteq G$ by Lemma~\ref{subBorel} \rm(ii) and \rm(iii), and hence $\rho(G/\Omega)<\frac{\sqrt{2}}{2}$.
		We thus may assume $G_\omega$ is not a two-group and hence $|G_\omega|\geqslant 5.$
		Therefore,
		\begin{equation*}
			\begin{aligned}
				\frac{1}{\sqrt{2}}\rho(G/\Omega) &=\text{max}\{\frac{|S|}{|G_\omega|}/\sqrt{\frac{|\Omega|}{2}}\}=\text{max}\{|S|\sqrt{\frac{2}{|G||G_\omega|}}\}\\&<\sqrt{\frac{2}{5q^2(q-1)(q^2+1)}}\frac{2q(q^2+1)}{q-1}
				=\sqrt{\frac{8(q^2+1)}{5(q-1)^3}}<1,
			\end{aligned}
		\end{equation*}
		where the last inequality holds by $q\geqslant 8.$
		This completes the proof.
	\end{proof}

	\begin{lemma}\label{q-r+1}
		Assume $G_\omega$ is a subgroup of $\ZZ_{q-r+1}{:}\ZZ_4.$ Then $|S| \leqslant q^2$ for any intersecting subset $S\subseteq G,$ and $\rho(G/\Omega)<\frac{\sqrt{2}}{2}.$
		In particular, if $G_\omega$ contains an element of order $4,$ then Sylow $2$-subgroups of $G$ are maximum intersecting subsets.
	\end{lemma}
	
	\begin{proof}
		Let $H$ be a maximal subgroup of $G$ isomorphic to $\ZZ_{q-r+1}{:}\ZZ_4.$
		Then
		\begin{equation*}
			D(G, [G,H])=(A_0\backslash 1)^G \cup (A_1\backslash 1)^G \subseteq D(G,\Omega).
		\end{equation*}
		Let $f$ be a function from $G$ to $\mathbb{R}$ such that $f(g)=1$ for $g\in (A_0\backslash 1)^G,$ $f(g)=\frac{2(q^2+q+2)}{q^2-q+r}$ for $g\in (A_1\backslash 1)^G,$ and $f(g)=0$ for remaining elements in $G.$
		Then $f$ is a $(G,\Omega)$-compatible class function.
		By Lemmas~\ref{lem:hoff} and \ref{eign}, the spectrum of $M^f$ is
		\begin{equation*}
			\begin{aligned}
				&\lambda_{\mathbf{1},f}=q^5-q^4+q^3-2q^2, \\  &\lambda_{X,f}=\lambda_{X_i,f}=-q^2, \\
				&\lambda_{Y_j,f}=\frac{2q^4+2q^3+4q^2}{q^2-q+r}, \\ &\lambda_{Z_k,f}=0, \\
				&\lambda_{W_\ell,f}=\frac{1}{2}q^3r+\frac{1}{2}q^2r+qr, 
			\end{aligned}
		\end{equation*}
		where $1\leqslant i\leqslant \frac{q^2}{2}, 1\leqslant j \leqslant \frac{q+r}{4},$ $1\leqslant k \leqslant \frac{q-r}{4},$ and $\ell \in \{1,2\}.$
		In particular, the row sum of $M^f$ is $d=q^5-q^4+q^3-2q^2$ and the minimal eigenvalue of $M^f$ is $\tau=-q^2.$
		Thus, by Hoffman's bound (Lemma~\ref{lem:hoff}), for any intersecting subset $S\subseteq G,$
		\begin{equation*}
			|S|\leqslant \frac{|G|}{1-\frac{d}{\tau}}=\frac{q^2(q-1)(q^2+1)}{q^3-q^2+q-1}=q^2.
		\end{equation*}
		If $G_\omega$ is a two-group, then $|S|\leqslant q^2$ for any intersecting subset $S$, and $\rho(G/\Omega)<\frac{\sqrt{2}}{2}$ by Lemmas~\ref{subBorel} \rm(ii) and \rm(iii), as stated in the Lemma.
		Thus, we may assume $G_\omega$ is not a two-group, and hence $|G_\omega|\geqslant 5.$
		Therefore,
		\begin{equation*}
			\frac{1}{\sqrt{2}}\rho(G/\Omega)=\text{max}\{\frac{|S|}{|G_\omega|}/\sqrt{\frac{|\Omega|}{2}}\}=\text{max}\{|S|\sqrt{\frac{2}{|G||G_\omega|}}\}\leqslant \frac{\sqrt{2}q^2}{\sqrt{5q^2(q-1)(q^2+1)}}<1.
		\end{equation*}
	    
	    On the other hand, it is clear that if $G_\omega$ contains an element of order $4$, then a Sylow $2$-subgroup of $G$ is a maximum intersecting subset of $G.$
		
		This completes the proof.
	\end{proof}

	With the above series of lemmas, we are ready to prove Theorem~\ref{thm:Sz-EKR}.
	\subsection*{Proof of Theorem~\ref{thm:Sz-EKR}}
	Theorem~\ref{thm:Sz-EKR} holds by Lemmas~\ref{lem:d2q-1} to \ref{q-r+1}.
	
	\section{Various examples of large intersecting subsets}\label{sec:examples}

	We shall construct large intersecting subsets in various transitive permutation groups that are not quasiprimitive.

	\subsection{Examples from affine groups}\
	
	Some different permutation representations of solvable 2-transitive permutation groups give rise to some critical examples.
	
	\begin{construction}\label{nobo}
		{\rm
			For $q=p^{d}$ with $p$ prime, let $E=(\bb{F}_{q},+)$ and $F=(\bb{F}_{q^{2}},+)$ be the additive groups of the fields $\FF_q$ and $\FF_{q^2}$, and let $E^{\times}$ and $F^{\times}$ denote the multiplicative groups of $\FF_q$ and $\FF_{q^2}$, respectively.
			These define two affine groups
			\[\begin{array}{l}
				G=F{:}F^{\times}=\AGL(1,p^{2d})=\ZZ_p^{2d}{:}\ZZ_{p^{2d}-1},\\
				H=E{:}E^{\times}=\AGL(1,q)=\ZZ_p^d{:}\ZZ_{p^d-1}.
			\end{array}\]
			Viewing $H$ as a subgroup of $G$, let $\Omega=[G:H]$.
			Then $G$ is a transitive permutation group on $\Omega$.
			Let
			\[S=F{:}E^{\times}=\ZZ_p^{2d}{:}\ZZ_{p^d-1}.\]
		}
	\end{construction}

	\begin{lemma}
		Let $G$, $H$, $\Omega$, and $S$ be as defined in Construction~$\ref{nobo}$ with $p$ odd.
		Then $S$ is an intersecting subset, and
		\[\mbox{$\rho(G/\Ome)<1$, and $\lim\limits_{q\to \infty}\rho(G/\Ome)=1$.}\]
	\end{lemma}
	\begin{proof}
		Observe that $G$ is a Frobenius group with the Frobenius kernel $F$ and a Frobenius complement $F^\times$.
		It follows that $S$ is a Frobenius group too.
		Let $x$ be a non-identity element of $S$.
		Then either $|x|=p$, or $|x|$ divides $p^d-1$.
		If $|x|=p$, then $x$ is conjugate to an element of $H$ since $F^\times$ acts transitively on the non-identity elements of $F$ in $G=F{:}F^\times$ by conjugation.
		If $|x|$ divides $p^d-1$, then $x$ is conjugate to an element of $H$ by Hall's theorem.
		Therefore, $S$ is an intersecting subset.
		
		Finally, calculations show that
		\[\begin{array}{l}
			|\Ome|=\frac{|G|}{|H|}=\frac{q^{2}(q^{2}-1)}{q(q-1)}=q(q+1),\\
			\frac{|S|}{|H|}=\frac{q^{2}(q-1)}{q(q-1)}=q= \sqrt{q^{2}}<\sqrt{q(q+1)}=\sqrt{|\Omega|}.
		\end{array}\]
		Thus $\rho(G/\Ome)=\dfrac{|S|}{\sqrt{|\Omega|}|H|}={\sqrt{q^{2}}\over\sqrt{q(q+1)}}=\sqrt{{q\over q+1}}<1$, and further,
		\[\lim\limits_{q\to \infty}\rho(G/\Ome)=\sqrt{{q\over q+1}}=1.\]
		This proves the lemma.
	\end{proof}

	\subsection{Some exceptional examples}
	
	We next present a few exceptional examples, which show that the inequality $\rho(G/\Ome)<1$ is not always true.
	
	\begin{table}[H]
		\[\begin{array}{|lllll|}\hline
			G&H&S&|\Ome|&\rho(G/\Ome) \\ \hline
			5^2{:}\SL(2,3) & 5{:}4 & 5^2{:}\Q_8 &30& \sqrt{10\over3} \\
			5^2{:}(\SL(2,3).2)&5{:}(4{:}2)&5^2{:}(\Q_8.2)&30&\sqrt{10\over3} \\
			29^2{:}(\SL(2,5)\times7)&29{:}28&29^2{:}(\Q_8\times7)&29.30&\sqrt{58\over15} \\
			29^2{:}(\SL(2,5)\circ28)&29{:}(4\circ28)&29^2{:}(\Q_8\circ28)&29.30&\sqrt{58\over15} \\
			3^3{:}\A_4&3^2{:}2&3^3{:}2^2&18&\sqrt2\\ \hline
		\end{array}\]
		\caption{Groups with large intersecting subsets.}
		\label{larget}
	\end{table}
	
	We remark that the examples in rows~1, 2 and 5 were first obtained in \cite{MRS}.
	
	\begin{lemma}\label{base-groups}
		Let $G,H,S$ be as in Table~$\ref{larget}$, and let $\Omega=[G:H]$.
		Then $G$ is a transitive permutation group on $\Omega$, and $S$ is an intersecting subset such that
		\[\sqrt2\le\rho(G/\Ome)<2.\]
	\end{lemma}
	\begin{proof}
		Here we only treat the group in row~4 since similar arguments work for the other candidates.
		Let
		\[\begin{array}{l}
			G= {29}^{2}{:}(\SL(2,5)\circ {28})<\AGL(2,29),\\
			H={29}{:}(4\circ{28})=29{:}(2\times28),\\
			S={29}^{2}{:}(\Q_{8}\circ{28}).
		\end{array}\]
		We need to prove that each element of $S$ is conjugate to an element of $H$.
		
		First, it is well-known that $G$ is a 2-transitive permutation group of degree $29^2$, and
		hence all elements of $G$ of order 29 are conjugate.
		
		Let $H_{29'}=\ZZ_4\circ\ZZ_{28}$ be a Hall $29'$-subgroup of $H$, and
		let $S_{29'}=\Q_8\circ\ZZ_{28}$ be a Hall $29'$-subgroup of $S$ such that $H_{29'}<S_{29'}$.
		Write
		\[\begin{array}{l}
			H_{29'}=\l x\r\circ\l c\r=\ZZ_4\circ\ZZ_{28},\\
			S_{29'}=\l x,y\r\circ\l c\r=\Q_8\circ\ZZ_{28}.
		\end{array}\]
		
		The polynomial $t^{2}+1 \in \bb{F}_{5}[t]$ is the minimal polynomial of any order $4$ element of $\SL(2,5)$. Thus all elements of order $4$ in $\SL(2,5)$ are conjugate.
		It follows that all elements of $\l x,y\r$ of order 4 are conjugate in $G$ to $x$.
		Let $g\in S_{29'}$.
		If $g\in\l c\r$, then $g\in H_{29'}$.
		Assume that $g\notin\l c\r$.
		Then $g=zc^i$ where $z\in\l x,y\r$ is of order 4, and $g$ is conjugate to $xc^i\in H$.
		Therefore, each element of $S$ is conjugate to an element of $H$.
		So $S$ is an intersecting subset.
		
		Finally, by definition, we have $|\Omega|=|G|/|H|=29\times 30$, and thus
		$$|S|/|H|=29\times 2 > \sqrt{29\times 30}=\sqrt{|\Omega|}.$$
		It then follows that the parameters are as stated in row 4 of Table~\ref{larget}.
	\end{proof}

	\subsection{Examples from special $p$-groups}\
	
	We present examples of intersecting subsets with large order based on some special $p$-groups.
	The first example is from the so-called {\em Suzuki $2$-groups}.
	
	\begin{construction}\label{nobo_Suz}
		{\rm
			Let $T=\Sz(q)$ with $q=2^e$, a Suzuki simple group, and let $Q$ be a Sylow $2$-subgroup of $T$.
			Let
			\[\begin{array}{rcl}
				G&=&\N_T(Q)=Q{:}\l g \r \cong (\ZZ_2^e.\ZZ_2^e){:}\ZZ_{q-1},\\
				H&\cong&\ZZ_4, \ \mbox{a cyclic subgroup of $G$ of order 4,}\\
				\Omega&=&[G:H], \ \mbox{of size ${1\over4}q^2(q-1)$.}\\
			\end{array}\]
		}
	\end{construction}
	
	\begin{lemma}\label{lem:nobo_Suz}
		Let $G,H,\Omega$ and $Q$ be as defined in Construction $\ref{nobo_Suz}$.
		Then $Q$ is an intersecting subset of $G$, and $\rho(G/\Omega)={q\over2\sqrt{q-1}}$.
	\end{lemma}
	\begin{proof}
		Note that each element of $Q$ is conjugate to an element of $G_\o=H\cong\ZZ_4$ in $G$, and hence $Q$ is an intersecting subset of $G$.
		Since $|Q|=q^2$, we have that
		$$\rho(G/\Omega)={|Q|\over|H|\sqrt{|\Omega|}}={q^2\over4\sqrt{{1\over4}q^2(q-1)}}
		={q^2\over2q\sqrt{q-1}}={q\over2\sqrt{q-1}}.
		$$
		This completes the proof.
	\end{proof}

	The next example is based on Sylow $p$-subgroups of simple unitary groups.
	
	\begin{construction}\label{nobo_PSU}
		{\rm
			Let $T=\PSU_3(q)$, where $q=p^f$ is odd, and $\gcd(3,q+1)=1$, and let $Q\cong \ZZ_p^f{:}\ZZ_p^{2f}$ be a Sylow $p$-subgroup of $T$.
			Let $x,y$ be two non-commuting element in $Q,$ and let
			\[\begin{array}{rcl}
				G&=&\N_T(Q)\cong (\ZZ_p^f.\ZZ_p^{2f}){:}\ZZ_{q^2-1},\\
				H&=& \l x, y\r \cong\ZZ_p.\ZZ_p^2,\\
				\Omega&=&[G:H],\ \mbox{of size ${q^3(q^2-1)\over p^3}$}.
			\end{array}\]
		}
	\end{construction}

	\begin{lemma}\label{lem:nobo_PSU}
		Let $G, H, \Omega$ and $Q$ be as defined in Construction~$\ref{nobo_PSU}$.
		Then $Q$ is an intersecting subset of $G$, and
		$\rho(G/\Omega)=\frac{q^3}{\sqrt{p^3q^3(q^2-1)}}.$
	\end{lemma}
	\begin{proof}
		Let $V=\mathrm{span}\{u,v,w\}$ be a unitary space, where $(u,v)$ is a hyperbolic pair and
		$w\in \mathrm{span}\{u,v\}^{\perp}$.
		Assume that $T=\PSU(V)=\PGU_3(q)$, consisting of unitary linear transformations over $V$.
		Let $G=T_{\mathrm{span}\{u\}}$, the stabilizer of the subspace span$\{u\}$.
		For convenience, we denote the matrix
		\begin{equation*}
			\begin{pmatrix}
				1 & a & -b^q\\
				0 & 1 & 0 \\
				0 & b & 1
			\end{pmatrix}
		\end{equation*}
		as $M(a,b).$ Then
		\[ Q= \{M(a,b) \mid  a,b\in \mathbb{F}_{q^2},a+a^q+bb^q=0 \}.\]
		Moreover, $G=\N_T(Q)=Q{:}\l g \r$, where $g=\mathrm{diag}(\lambda, \lambda^{-q}, \lambda^{q-1}),$
		with $\lambda$ being a generator of the multiplicative group of $\mathbb{F}_{q^2}$.
		
		One may easily verify the following two matrix equations
		\begin{equation}\label{e:PSU+}
			M(a,b)M(a',b')=M(a+a'-b^qb',b+b'),
		\end{equation}
		\begin{equation}\label{e:PSUc}
			M(a,b)^{\diag(\lambda, \lambda^{-q}, \lambda^{q-1})}=M(\lambda^{-(q+1)}a,\lambda^{1-2q}b).
		\end{equation}
		Then (\ref{e:PSU+} ) implies $$\C_Q(M(a,b))=\{M(a',b')\mid b^qb'=(b')^qb,a'+(a')^q+b'(b')^q=0\}.$$ Note that $b^qb'=(b')^qb=(b^qb')^q$ is equivalent to $b^qb'\in \FF_q.$ Thus, $M(a,0)\leqslant \Z(Q)$ for any $a\in\FF_{q^2}$ with $a+a^q=0,$ and $|\C_Q(M(a',b'))|=q^2$ for any $a',b'\in \FF_{q^2}^{\times}$ with $M(a',b')\in Q.$ This implies
		\begin{equation}\label{e:PSUZ}
			\Z(Q)=\{M(a,0)\mid a\in \FF_q^2, a+a^q=0\}\cong \ZZ_p^f,~\mbox{and}~ Q/\Z(Q)\cong \FF_q^{+} \cong \ZZ_p^{2f}.
		\end{equation}
		Further, two elements of $\Z(Q)$ of order $p$ are conjugate by an element in $\l g \r$ by  (\ref{e:PSUc}).
		
		Take an element $x\in Q\backslash \Z(Q)$.
		Then $x=M(a,b)$ with $b\in \FF_{q^2}^{\times}$ and $a+a^q+bb^q=0.$
		Note that $\l g \r$ acts regularly by conjugation on $Q/\Z(Q)$ by (\ref{e:PSUc}) and (\ref{e:PSUZ}).
		Thus $\C_G(x)\leqslant Q$, and $|\C_G(x)|=q^2$ .
		It follows that
		$$|x^G|=|G|/|\C_G(x)|=\frac{q^3(q^2-1)}{q^2}=q(q^2-1),$$
		and hence $x^G=Q\backslash \Z(Q),$ refer to \cite{Dornhoff}.
		That is to say, non-identity elements of $Q$ form two conjugacy classes of $G$, one of which is $\Z(Q)\setminus\{1\}$, and the other is $\Q\setminus\Z(Q)$.
		
		By definition, $H$ is non-abelian. Further, $Q/\Z(Q)$ is abelian implies $[x,y]\in \Z(Q)$ and $H\cong \ZZ_p.\ZZ_p^2$.
		Then $\Z(H)=H\cap\Z(Q)\cong \ZZ_p$, and $H\cap(Q\setminus\Z(Q))\not=\emptyset$.
		Thus each element in $\Z(Q)\setminus\{1\}$ is conjugate to some element of $\Z(H)$, and each element of $Q\setminus\Z(Q)$ is conjugate to some element of $H\setminus \Z(H)$.
		Therefore, $Q$ is an intersecting subset of $G$.
		Then
		$$\rho(G/\Omega)=\frac{|Q|}{|H|\sqrt{|\Omega|}}=\frac{|Q|}{\sqrt{|G||H|}} =\frac{q^3}{\sqrt{p^3q^3(q^2-1)}},$$
		completing the proof.
	\end{proof}

	\subsection*{ Proof of Theorem~\ref{no-bound}}\
	
	For any positive number $M$. There exists $e_0$ such that whenever $q=2^e\geqslant 2^{e_0}$ we have that ${q\over2\sqrt{q-1}}>M.$
	There also exists a prime $p$ with $\gcd(3,p+1)=1$ and an integer $f_0>3$ such that whenever $q=p^f\geqslant p^{f_0}$ we have $q+1$ is coprime to $3$ and $\frac{q^3}{\sqrt{p^3q^3(q^2-1)}}>M.$
	
	Then the proof of Theorem~\ref{no-bound} follows from Construction~\ref{nobo_Suz} , Lemma~\ref{lem:nobo_Suz}, Construction~\ref{nobo_PSU} and Lemma~\ref{nobo_PSU}.
	\qed

	\bibliographystyle{plain}


\end{document}